\documentclass[11pt]{amsart}
\usepackage{amscd}
\usepackage{amsmath}
\usepackage{amssymb}
\usepackage{latexsym}
\usepackage{xypic}

\input
xy \xyoption{all}

\newtheorem{thm}{Theorem}[section]
\newtheorem{prop}[thm]{Proposition}
\newtheorem{lem}[thm]{Lemma}
\newtheorem{sublem}[thm]{Sublemma}
\newtheorem{lem-def}[thm]{Lemma-Definition}
\newtheorem{cor}[thm]{Corollary}
\theoremstyle{remark}
\newtheorem{ex}[thm]{Example}
\newtheorem{rmk}{Remark}[section]
\theoremstyle{definition}
\newtheorem{dfn}{Definition}[section]

\numberwithin{equation}{section}

\newcommand{\frakg}{{\mathfrak g}}

\newcommand{\bbG}{{\mathbb G}}

\newcommand{\bbL}{{\mathbb L}}

\newcommand{\calA}{{\mathcal A}}
\newcommand{\calB}{{\mathcal B}}
\newcommand{\calC}{{\mathcal C}}
\newcommand{\calD}{{\mathcal D}}

\newcommand{\calF}{{\mathcal F}}
\newcommand{\calG}{{\mathcal G}}
\newcommand{\calH}{{\mathcal H}}
\newcommand{\calI}{{\mathcal I}}

\newcommand{\calK}{{\mathcal K}}
\newcommand{\calL}{{\mathcal L}}
\newcommand{\calM}{{\mathcal M}}

\newcommand{\calO}{{\mathcal O}}
\newcommand{\calP}{{\mathcal P}}
\newcommand{\calQ}{{\mathcal Q}}
\newcommand{\calR}{{\mathcal R}}

\newcommand{\calT}{{\mathcal T}}

\newcommand{\calX}{{\mathcal X}}
\newcommand{\calY}{{\mathcal Y}}
\newcommand{\calZ}{{\mathcal Z}}
\newcommand{\End}{{\mathrm{End}}}

\newcommand{\Hom}{{\mathrm{Hom}}}
\newcommand{\Ext}{{\mathrm{Ext}}}
\newcommand{\Lie}{{\mathrm{Lie}}}
\newcommand{\pic}{{\mathrm{Pic}}}
\newcommand{\pr}{{\mathrm{pr}}}

\newcommand{\spec}{{\mathrm{Spec}}}

\newcommand{\Mod}{{\mathrm{-Mod}}}
\newcommand{\GL}{{\bbG\bbL}}
\newcommand{\gl}{{\mathfrak g\mathfrak l}}

\newcommand{\nc}{\newcommand}
\nc{\on}{\operatorname}

\title[The 2-group of auto-equivalences of abelian categories]{The 2-group of linear auto-equivalences of an abelian category and its Lie 2-algebra}\thanks{Supported by DARPA through the grant HR0011-09-1-0015.}
\begin{document}

\author{Xinwen Zhu}

\address{Department of
Mathematics, Harvard University, Cambridge, MA 02138, USA}
\email{xinwenz@math.harvard.edu}
\date{September, 2009}
\maketitle

\begin{abstract}For any abelian category $\calC$ satsifying (AB5) over a separated, quasi-compact scheme $S$ , we construct a stack of 2-groups $\GL(\calC)$ over the flat site of $S$. We will give a concrete description of $\GL(\calC)$ when $\calC$ is the category of quasi-coherent sheaves on a separated, quasi-compact scheme $X$ over $S$. We will show that the tangent space $\gl(\calC)$ of $\GL(\calC)$ at the origin has a structure as a Lie 2-algebra.
\end{abstract}

\tableofcontents

\section*{Introduction}

In nowadays mathematics, much of the research on certain algebraic
object is to study its representation theory. On the one hand, the
representation theory of an algebraic object reveals some of its
profound structures hidden underneath. A good example is that the structure
of a complex semi-simple Lie algebra is much revealed via its
representation theory. On the other hand, the representation spaces
themselves usually contain significant meanings. For example, the space of $L^2$ functions on an adelic group, regarded as the
left regular representation of the group, plays a
fundamental role in harmonic analysis and arithmetics.

Let us recall the classic set-up of a representation of an algebraic
structure. Let $A$ be an algebraic structure. For example, $A$
could be a group, or an algebra, or a Lie algebra. A
representation usually involves a vector space $V$ and a
homomorphism from $A$ to $\End(V)$ or $GL(V)$, preserving the
algebraic structure. In any case, our algebraic structure acts on a \emph{vector space}.

In recently years, there are many clues that the basic set-up of representation theory must be widen. Namely, we should allow the representation spaces not only to be vector spaces, but also to be categories (or even higher categories). There are two main motivations for this. The first comes from the procedure so-called categorification (or known as geometrization). This procedure usually involves replacing the original representation space $V$ by an abelian (or additive) category $\calC$ on which the algebraic structure acts, and whose Grothendieck group $K(\calC)$, after tensoring the field coefficient, gives rise back to the original vector space. The second motivation is that the algebraic structure itself is so complicated that the classical representation theory of the algebraic structure is probably just not the right thing to consider. For instance, in the local geometrical Langlands correspondence, the de Rham local system on the punctured disc (the Langlands parameters) should parameterize certain representations of the loop group, which at first thought should be (smooth) representations of the group. However, there are just too few such representations. It turn out the right thing to parameterize is the abelian categories with an action of the loop group (see \cite{FG,Fr} for an explicit proposal of these categories). Another example is the representation theory of double loop groups and algebras (cf. \cite{FZ1,FZ2}). As explained in \emph{loc. cit.}, if one insists on studying their representations in the classical sense, he will encounter horrible infinity. On the other hand, such infinity is in some sense absorbed in a cohomology class lying in the third group cohomology. To realize such class, one should require the double loop group acts (gerbally) on some abelian category (see \cite{FZ1} for a description of this category).

\medskip

This paper is aimed to explain how  the idea of representations on categories could work for algebraic groups and Lie algebras. To this end, let us first review the corresponding classical set-up.

Let $V$ be a (finite) dimensional vector space over a field $k$. Let
$GL(V)$ be the group of linear automorphisms of $V$. Let us recall
$GL(V)$ is an algebraic group. In particular, it is a sheaf of
groups over the $(\mathbf{Aff}/k)_{fppf}$. The functor it represents
is as follows. For any $k$-algebra $A$ , $GL(V)(A)$ is the group of
$A$-linear automorphisms of the free $A$-module $A\otimes V$. Now,
an algebraic representation of an algebraic group $G$ over $k$ on
$V$ is (up to conjugation) just an algebraic group homomorphism from
$G$ to $GL(V)$. A projective representation of $G$ on $V$ then is a
homomorphism from $G$ to $PGL(V)$. Furthermore, the tangent space of
$GL(V)$ at the unit, usually denoted by $\gl(V)$, is naturally a Lie
algebra. A representation of a Lie algebra $\frakg$ on $V$ then is
just a Lie algebra homomorphism from $\frakg$ to $\gl(V)$.

In this paper, we will write down a similar construction, with $V$
replaced by a $k$-linear abelian category $\calC$ that satisfies
Grothendieck's axiom (AB5). Namely, we will define $\GL(\calC)$, the
"automorphism group" of $\calC$, as the category of $k$-linear
auto-equivalences of $\calC$.  It has a natural monoidal structure,
which makes it into a 2-group. The first goal of this paper is to
give it a structure as a sheaf of 2-groups over
$(\mathbf{Aff}/k)_{fppf}$. This is possible because when $\calC$
satisfies (AB5), it makes sense to talk about the base change
$\calC_A$ of $\calC$ to any commutative $k$-algebra $A$, and this
construction will give a sheaf of abelian categories over
$(\mathbf{Aff}/k)_{fppf}$. Now define the $A$-points of $\GL(\calC)$
to be the groupoid of $A$-linear auto-equivalences of $\calC_A$.
Once one can define, for every $k$-algebra homomorphism $f:A\to B$,
a monoidal pullback $f^*:\GL(\calC)(A)\to\GL(\calC)(B)$ satisfying
the usual compatibility conditions, he will deduce that $\GL(\calC)$
is indeed a sheaf of 2-groups over $(\mathbf{Aff}/k)_{fppf}$. This
is indeed the case, as will be shown in \S \ref{2-grp of
auto-equiv}.
\begin{thm}Let $\calC$ be a $k$-linear abelian category $\calC$ that satisfies Grothendieck's axiom (AB5). Then $\GL(\calC)$ is a sheaf of 2-groups over $(\mathbf{Aff}/k)_{fppf}$.
\end{thm}

The first example of this construction is when $\calC=\on{Qcoh}(X)$
is the category of quasi-coherent sheaves over a quasi-compact,
separated $k$-scheme $X$. We will show in \S \ref{example} that
$\GL(\on{Qcoh}(X))$ is a semi-direct product of the Picard stack of
$X$ and the sheaf of $k$-automorphisms of $X$.

After endowing $\GL(\calC)$ with a structure of 2-groups over $(\mathbf{Aff}/k)_{fppf}$, we can study it's tangent space (i.e. tangent stack) at the unit. Let us first recall that if $G$ is an algebraic group over $k$, then its tangent space at the unit has a natural Lie algebra structure. However, if $G$ is only a group functor over $\mathbf{Aff}/k$, it is a subtle issue to give $\Lie G:=\ker(G(k[\varepsilon]/\varepsilon^2)\to G(k))$ a Lie algebra structure. In \cite{Dem}, certain conditions were proposed to guarantee that $\Lie G$ has a natural Lie algebra structure. A group $G$ that satisfies these conditions is called \emph{good} in \emph{loc. cit.} Let us remark that these conditions are satisfied if $G$ takes product of rings to product of groups. Therefore, all algebraic groups are good.

Now let $\calG$ be a sheaf of 2-groups over $(\mathbf{Aff}/k)_{fppf}$\footnote{All the discussions apply to presheaves of 2-groups.}. The conditions given in \cite{Dem} have natural generalizations to this case (cf. \S \ref{2-grps to Lie 2-algs}, Definition \ref{Cond(E)} and Definition \ref{good}). Assume that $\calG$ is good, i.e. it satisfies these conditions, then
\[\Lie(\calG):=\ker(\calG(k[\varepsilon]/\varepsilon^2)\to\calG(k)),\]
which a priori is a stack, will carry on some additional structures,
which will be the categorical generalization of the usual Lie
algebras. A stack together with these structures is called a Lie
2-algebra, which was originally introduced in \cite{BC} by Baez and
Crans. We remark that to make the above statement true, we have to
slightly generalize their original definition of the Lie 2-algebra
(see \S \ref{dic2} for the detailed definitions).

Now we can state the main result of this paper.

\begin{thm}\label{main} For any $k$-linear abelian category $\calC$ satisfying (AB5), the sheaf of 2-groups $\GL(\calC)$ is good, and therefore $\gl(\calC):=\Lie(\GL(\calC))$ has a natural Lie 2-algebra structure.
\end{thm}
Observe that in contrast to $GL(V)$, $\GL(\calC)$ is in general not algebraic. Therefore that $\GL(\calC)$ is good is not automatic.

We can somehow describe the addition of $\gl(\calC)$ (a Lie 2-algebra a priori is a Picard groupoid, see \S \ref{dic2} for details).  We will show that every object $F$ in $\gl(\calC)$ will assign every object $X$ in $\calC$ a self-extension $F(X)$ in $\calC$, and the self-extension of $X$ given by the sum $F+G$ in $\gl(\calC)$ will be canonically isomorphic to the Baer sum of $F(X)$ and $G(X)$. In other words, there is a homomorphism of Picard groupoids
\[\gl(\calC)\to\Ext(\on{Id}_\calC,\on{Id}_\calC),\]
where $\Ext(\on{Id}_{\calC},\on{Id}_{\calC})$ is the Picard groupoid of self-extensions of the identity functor of $\calC$. Therefore, it appears at the first sight strange since an object in $\gl(\calC)$ will not act on $\calC$ itself. But arguably even in classical case, $\gl(V)$ does not act on $V$. Indeed, $\gl(V)$ is a map from $V$ to $T_0V\cong V$.

In the above discussion, we assumed that $\calC$ is $k$-linear.
However, in the main body of the paper, we will work on those
$\calC$ over a separated, quasi-compact scheme $S$.

\medskip

In \cite{FZ1}, E. Frenkel and the author introduced the notion of "gerbal
representation" of a (discrete) group on an abelian category. Now we can generalize this notion to algebraic groups and Lie algebras. It will be used in \cite{FZ2} in the study of double loop Lie algebras.

Let $\pi_0(\GL(\calC))$ denote the coarse moduli space of $\GL(\calC)$, i.e. the sheafification of the pre-sheaf which assigns every $A$ the set of isomorphism classes of objects in $\GL_A(\calC_A)$. Then $\pi_0(\GL(\calC))$ is a group over $\mathbf{Aff}/k$. Let $H^0(\gl(\calC))$ denote the set of isomorphism classes of objects in $\gl(\calC)$. This is naturally a $k$-Lie algebra.

\begin{dfn}Let $\calC$ be a $k$-linear abelian category satisfying (AB5). Then
\begin{enumerate}
\item An algebraic representation of a $k$-group $G$ on $\calC$ is a $k$-2-group homomorphism $G\to\GL(\calC)$ (up to conjugacy);

\item A gerbal representation of $G$ on $\calC$ is a $k$-group homomorphism $G\to\pi_0(\GL(\calC))$;

\item A representation of a Lie algebra $\frakg$ on $\calC$ is a Lie 2-algebra homomorphism $\frakg\to\gl(\calC)$;

\item A gerbal representation of $\frakg$ on $\calC$ is a Lie algebra homomorphism $\frakg\to H^0(\gl(\calC))$.
\end{enumerate}
\end{dfn}

\begin{rmk}
(i) As shown in \cite{FZ1,FZ2}, gerbal representations realize third cohomology
classes. Explicit examples of such representations for double loop groups
have been constructed in \cite{FZ1}, and for double loop Lie algebras such
representations will be constructed in \cite{FZ2}.

(ii) In \cite[\S 20.1]{FG}, another definition of algebraic
representation of an algebraic $k$-group $G$ on $\calC$ is given. It
is not hard to see that these two definitions coincide.
\end{rmk}

\medskip

\noindent\emph{Plans of the paper.} The plan of this paper is as follows.

In \S \ref{review of 2-grp}, we review the basic theory of
2-groups, which can be regarded as group objects in the
2-category of categories.

Recall that on a scheme $S$, there are three successive abelian categories of sheaves
$\on{Ab}\supset\on{Mod}(\calO_S)\supset\on{Qcoh}(S)$.
While a stack of abelian categories over $S$ is the categorical analogue of a sheaf of abelian groups on $S$, it is desire to also have a similar categorical analogue of a quasi-coherent sheaf on $S$. Such an analogue does exist, and was addressed in \cite{G}. \S \ref{sheaf of abelian cat} is a review of this notion.

Now given a quasi-coherent sheaf of abelian categories $\calC$, we construct a stack of 2-groups $\GL(\calC)$ in \S \ref{2-grp of auto-equiv}, whose $\spec A$-points are just $A$-linear auto-equivalences of $\calC_A$. As a byproduct, we obtain a localization functor from the 2-category of $R$-linear abelian categories to the 2-category of quasi-coherent sheaves of abelian categories over $\spec R$. We will also discuss part of the center of $\GL(\calC)$.

In \S \ref{example}, we work out an explicit case of $\GL(\calC)$.
Namely, when $\calC$ is the category of quasi-coherent sheaves on a
separated, quasi-compact scheme $X$ over $S$. It turns out in this
case, $\GL(\calC)$ is a semi-direct product of the Picard stack of
$X$ and the sheaf of automorphisms of $X$ over $S$. We also propose
a description of $\GL(\calC)$ when $\calC$ is the category of
twisted quasi-coherent sheaves on $X$.

In \S \ref{Lie 2-alg gl(C)}, we begin our linearization procedure to the 2-group $\GL(\calC)$. It \S \ref{2-grps to Lie 2-algs}, we give some general conditions of a stack of 2-groups so that its tangent stack at the original admits a Lie 2-algebra structure. The main results of the paper is given in \S \ref{gl(C)}, where we prove that the 2-group $\GL(\calC)$ satisfies the these general conditions. Therefore, we associate $\calC$ a Lie 2-algebra.

Appendix (\S \ref{dic}) is devoted to giving detailed definitions of strictly commutative Picard stacks and Lie 2-algebras in a ringed topos, as well as their homological interpretations. 

\medskip

\noindent\emph{Conventions and Notations.} We will use the following conventions throughout the paper. If
$\calC$ is a category, $x\in\calC$ will denote an object in
$\calC$, $(f:x\to y)\in\calC$ will denote a morphism in $\calC$.
In an abstract monoidal category, we will write $xy$ in stead of
$x\otimes y$ for the tensor product. An expression like $xyzw$ is
understood as $((xy)z)w$, etc. If the monoidal category is
symmetrical monoidal, sometimes $xy$ is also written as $x+y$ and
the unit object is denoted by $0$.

In the paper, we will use the terminology stack and sheaf of categories interchangably.

\medskip

\noindent\emph{Acknowledgement.} This paper was motivated by my
joint project with Edward Frenkel \cite{FZ1,FZ2} on gerbal
representations of double groups and Lie algebras. I would like to
express my deep gratitude to him for collaboration and numerous
discussions. I would also like to thank Martin Olsson and Chenyang
Xu for useful discussions.

\section{2-groups}\label{review of 2-grp}
\subsection{The 2-category of 2-groups} We recall the definition of 2-groups. A good
introduction for this subject is \cite{BL}.

\begin{dfn}    \label{2-grp}
A 2-{\em group} is a monoidal groupoid $\calG$ such that the set
of isomorphism classes of objects of $\calG$, denoted by
$\pi_0(\calG)$, is a group under the multiplication induced from
the monoidal structure. Let $I_{\calG}$ (or $I$ for brevity)
denote the unit object of $\calG$. We set $\pi_1(\calG)=\End_\calG
I$.
\end{dfn}

In the literature, these objects often appear under different
names. For example, they are called weak 2-groups in \cite{BL},
and are called gr-categories in \cite{Br,S}.

It is clear that any group (in the usual sense) can be regarded as
a 2-group with the trivial $\pi_1$. All 2-groups form a (strict)
2-category, with objects being 2-groups, 1-morphism being the
1-homomorphisms between 2-groups (i.e. monoidal functors), and
2-morphism being the monoidal natural transformations of monoidal
functors.

Let $F:\calH\to\calG$ be a 1-homomorphism of 2-groups. Then $\ker F$
is defined to be the category whose objects are the pairs
$(x,\iota)$, where $x\in\calH$ and $\iota:F(x)\cong I_{\calG}$, and
whose morphisms from $(x,\iota)$ to $(x',\iota')$ consisting of
morphisms  $f:x\to x'$ in $\calH$ such that $\iota'F(f)=\iota$. It
is easy to see that $\ker F$ is a 2-group, and the functor $i_F:\ker
F\to \calH$ sending $(x,\iota)$ to $x$ is a faithful 2-group
1-homomorphism. In other words, $\ker F$ is the fiber product
$\calH\times_{\calG}I_{\calG}$, where $I_\calG$ is regarded as a set
with one element, and therefore as a groupoid.

A sequence of 2-group 1-homomorphisms
\[1\to\calH\stackrel{i}{\to}\calG\stackrel{p}{\to}\calK\to 1\]
is called exact if: (i) $p$ is essentially surjective; and (ii)
there is a 1-isomorphism $j:\calH\cong\ker p$ and a 2-isomorphism
$i_p j\cong i$. The sequence is called split if there is a 2-group
1-homomorphism $s:\calK\to\calG$ such that $ps\cong\on{Id}_\calK$.

We recall that if the monoidal structure of a 2-group is upgraded
to a tensor category structure (i.e., there exists a commutativity
constraint whose square is the identity), then this 2-group is
called a Picard groupoid. Picard groupoids should be regarded as
commutative 2-groups. A 1-homomorphism $f:\calP_1\to\calP_2$
between two Picard groupoids is a 2-group 1-homomorphism that
respects the commutativity constraints. The category of
1-homomorphisms $\Hom_{\on{pic}}(\calP_1,\calP_2)$ between two
Picard groupoids forms a Picard groupoid (cf. \cite{Del} \S
1.4.7). More discussions on the Picard groupoid will be given in
\S\ref{dic}.

The main examples of 2-groups we concern in the paper are the
following.
\begin{ex}\label{GL(C)}
For $\calC$ an abelian category, we denote $\GL(\calC)$ the category
of auto-equivalences of $\calC$. By definition, the objects of
$\GL(\calC)$ are additive functors $x:\calC\to\calC$ which are
equivalences of categories. The morphisms $\Hom_{\GL(\calC)}(x,y)$
are the set of natural transforms from $x$ to $y$ which are
isomorphisms. It is clear from the definition that $\GL(\calC)$ is a
strict\footnote{We recall that a monoidal category is called strict
if the associativity constraints and the unit constraints are
identity maps.} monoidal category. Furthermore, $\GL(\calC)$ is a
2-group, with
$\pi_1(\GL(\calC))=\mathrm{Aut}_\calC\on{Id}_\calC=:\calZ(\calC)^\times$.
Here $\mathrm{Aut}_\calC\on{Id}_\calC$ is the group of automorphisms
of the identity functor of $\calC$.
\end{ex}

\begin{ex}\label{Picard} Let $X$ be a scheme over $S$. Denote $\pic_X$ the category of
invertible sheaves on $X$. It is easy to see that $\pic_X$ is a
2-group (in fact a strictly commutative Picard
groupoid\footnote{We recall that a Picard groupoid $\calP$ is
strictly commutative if the commutativity constraint $c$ satisfies
$c_{x,x}=\on{id}_x$ for any $x\in\calP$.}), with
$\pi_0(\pic_X)=H^1(X,\calO_X^*)$, the Picard group of $X$, and
$\pi_1(\pic_X)=H^0(X,\calO_X^*)$.
\end{ex}

\subsection{Coherent 2-groups}At the first sight of the definition of
2-groups, it seems that we lack an important structure in the
theory of 2-groups, compared with the usual group theory. Namely,
there is no functorial way to take the inverse of an object in the
2-group. However, as explained in \cite{BL}, this is not a serious
problem.

Let $\calG$ is a 2-group. Then for every $x\in\calG$, one can
choose $\sigma(x)\in\calG$ and an isomorphism $e_x:\sigma(x)x\cong
I$. From this, one obtains a canonical isomorphism
\begin{equation}\label{x=sigma2(x)}x\to Ix\to
(\sigma(\sigma(x))\sigma(x))x\to\sigma(\sigma(x))(\sigma(x)x)\to\sigma(\sigma(x))I\to\sigma(\sigma(x))
\end{equation}
and therefore, a canonical isomorphism
$i_x:x\sigma(x)\cong\sigma(\sigma(x))\sigma(x)\cong I$. Then
$(x,\sigma(x),i_x,e_x)$ form an adjunction quadruple. That is, the following maps (induced from the associativity constraints, the unit constraints and $i_x,e_x$)
\[x\to xI\to x(\sigma(x)x)\to (x\sigma(x))x\to Ix\to x\]
\[\sigma(x)\to \sigma(x)I\to \sigma(x)(x\sigma(x))\to (\sigma(x)x)\sigma(x)\to I\sigma(x)\to \sigma(x)\]
are identity maps.
A coherent 2-group by definition is a 2-group
$\calG$ such that each object $x\in\calG$ is equipped with an
adjunction quadruple. Coherent
2-groups naturally form a 2-category, with 1-homomorphisms being monoidal functors\footnote{The monoidal functors will automatically respect to adjunction quadruples.} and 2-homomorphisms being monoidal natural transformations. As explained in \emph{loc. cit.}, the natural forgetful
functor $\calF or$ from the 2-category of coherent 2-groups to the
2-category of 2-groups is an equivalence of 2-categories. In what
follows, 2-groups will always mean coherent 2-groups, i.e. we fix
a quasi-inverse of $\calF or$.

Let $\calG$ be a coherent 2-group, then there is a functorial way to
take the inverse. That is, one can upgrade $\sigma$ to an
auto-equivalence $\sigma:\calG\to\calG$, such that for any $f:x\to
y$, $e_x=e_y\circ(\sigma(f)f)$. (Clearly, such $\sigma(f)$ exists
and is unique.) Observe that $\sigma$ is not a 2-group homomorphism.
Rather, one can show that there exist functorial isomorphisms
$\sigma(xy)\cong \sigma(y)\sigma(x)$ compatible with the
associativity constraints and the unit constraints, so that we call
$\sigma$ an inversion 1-anti-homomorphism. Observe that the
isomorphism (\ref{x=sigma2(x)}) can be interpreted as there is a
canonical 2-isomrphism $e:\on{Id}\cong\sigma^2$.

If $\calG$ is a coherent 2-group, it makes sense to define the
conjugate action
\begin{equation}\label{conj}\on{C}:\calG\times\calG\to\calG,\quad \on{C}(x,y)=xy\sigma(x)\end{equation}
For brevity, sometimes, we simply denote $\on{C}(x,y)$ by $y^x$.
It also makes sense to define the commutator functor:
\begin{equation}\label{comm}\on{comm}:\calG\times\calG\to\calG, \quad \on{comm}(x,y)=xy\sigma(x)\sigma(y)\end{equation}
For brevity, sometimes, we simply denote $\on{comm}(x,y)$ by
$(x,y)$.

We will need a few statements, which will be used in
\S\ref{2-grps to Lie 2-algs}. First, we recall the following well-known Coherence Theorem for 2-groups.

\begin{prop}\label{coherence}(\bf Coherence Theorem\rm) Let $\calG$ be a coherent 2-group, and $x,y$ are two objects in $\calG$. Then all the
isomorphisms from $x$ to $y$ that are compositions of various constraints of
the 2-group (i.e. the associativity constraints, the unit
constraints, and the constraints of $\sigma$) are the same.
\end{prop}

The following two lemmas are the direct consequence of this proposition.
\begin{cor}\label{for (i)} Let $\calG$ be a 2-group. For
$x,y\in\calG$, the $s_{x,y}$ be the canonical isomorphism
\[s_{x,y}:(x,y)\to\sigma^2(x,y)\to\sigma(\sigma^2(y)\sigma^2(x)\sigma(y)\sigma(x))\leftarrow\sigma(yx\sigma(y)\sigma(x))=\sigma(y,x)\]
Then the following diagram commutes.
\[\xymatrix{(x,y)\ar^{s_{x,y}}[r]\ar@{=}[d]&\sigma(y,x)\ar^{\sigma(s_{y,x})}[d]\\
(x,y)\ar^{e}[r]&\sigma^2(x,y) }\]
\end{cor}

\begin{cor}\label{tri-identity}Let $\calG$ be a 2-group. Define
\begin{equation}\label{tricomm}(x,y,z)=((x,y),z^y)((y,z),x^z)((z,x),y^x)
\end{equation}
Then there is a canonical isomorphism $(x,y,z)\cong I$. In
addition, the isomorphism $(x,y,z)\cong I$ and $(y,z,x)\cong I$
will induce a canonical isomorphism from $((x,y),z^y)$ to itself
in a way the same as (\ref{x=sigma2(x)}). Then this isomorphism is
the identity map.
\end{cor}
\begin{proof}The first part of the corollary is a direct consequence of the Coherence Theorem. To prove the second part, let $X=((x,y),y^z)$ and $Y=((y,z),x^z)((z,x),y^x)$. Observe that the statement then is equivalent to saying that the two isomorphisms
\[(XY)X\cong IX\cong X\]
\[(XY)X\cong X(YX)\cong XI\cong X\]
are the same. This follows again from the Coherence Theorem.
\end{proof}

Finally, we also need

\begin{lem}\label{for (ii)}Let $\calP$ be a strictly commutative Picard groupoid.
Let $x,y\in\calP$ and $i:xy\cong I$ and $j:yx\cong I$. We obtain a
canonical isomorphism $x\to x$ in a way the same as
(\ref{x=sigma2(x)}). Then this isomorphism is the identity map if
and only if the following diagram commutes
\[\xymatrix{xy\ar^{c_{x,y}}[rr]\ar_{i}[dr]&&yx\ar^{j}[dl]\\
&I& }\] where $c_{x,y}$ is the commutativity constraint.
\end{lem}
\begin{proof}Let us use $c$ to denote
the commutativity constraints in this paragraph.

The fact that the isomorphism $x\to x$ is the identity map is
equivalent to the commutativity of the following pentagon.
\[
\xymatrix{(xy)x\ar[rr]\ar_{ix}[d]&&x(yx)\ar^{xj}[d]\\
             Ix\ar[dr]\ar^{c_{I,x}}[rr]&& xI\ar[dl]\\
             &x&\ .}
\]
Since the lower triangle in the diagram is commutative, the
commutativity of the pentagon is equivalent to the commutativity
of the upper square.

The hexagon axiom together with the fact $c_{x,x}=\on{id}_x$ gives
rise to the commutative diagram
\[
\xymatrix{(xy)x\ar[rr]\ar[dr]_{c_{x,y}x}&&x(yx)\ar^{c_{x,yx}}[dl]\\
&(yx)x&\ .}
\]

Finally, from the following diagram, it is clear that the
commutativity of the outer square is equivalent to the
commutativity of the left triangle.
\[
\xymatrix{(xy)x\ar[rr]\ar_{ix}[dd]\ar_{c_{x,y}x}[dr]&&x(yx)\ar^{xj}[dd]\ar^{c_{x,yx}}[dl]\\
                   &(yx)x\ar_{jx}[dl]&\\
             Ix\ar^{c_{I,x}}[rr]&& xI.}
\]
\end{proof}

\subsection{Stacks of 2-groups}

As usual, we could work in any topos $\calT$ instead of the
category of sets. Then a sheaf of 2-groups (or a stack of 2-groups)
$\calG$ will be a stack over $\calT$, together with morphisms
\[I:\calT\to\calG, \quad\quad \otimes:\calG\times_\calT\calG\to\calG\]
satisfying the
associativity and the unit constraints, such that for each
$U\in\calT$, the induced structure on $\calG(U)$ is a
2-group. As explained in the previous subsection, we will also assume a stack of 2-groups contains a datum $\sigma:\calG\to\calG$, together with an isomorphism
\[e:\sigma\otimes\on{Id}\cong I\pi:\calG\to\calG,\]
where $\pi:\calG\to\calT$ is the structural map. In practice, it is convenient to think $\calG$ is a sheaf of groupoids over $\calT$ such that for any $U\in\calT$, $\calG(U)$ is a (coherent) 2-group and the pullback
functor respects to the monoidal structure (i.e. for $f:V\to U$,
$g:W\to V$, $f^*,g^*$ are monoidal functors, and the canonical
isomorphism $g^*\circ f^*\cong (f\circ g)^*$ are monoidal natural
transforms). Denote by $I_U$ the unit object in $\calG(U)$. Observe
that $U\mapsto\mathrm{End}_{\calG(U)}(I_U)$ is a sheaf of abelian
groups over(= an abelian group in) $\calT$, which is denoted by
$\pi_1(\calG)$. However, $U\mapsto\pi_0(\calG(U))$ is usually only
a presheaf. We will denote its \emph{sheafification} by
$\pi_0(\calG)$ (so in general
$\pi_0(\calG)(U)\neq\pi_0(\calG(U))$). This is a sheaf of groups
over $\calT$, called the coarse moduli of $\calG$. If one regards
$\pi_0(\calG)$ as a 2-group, then the natural projection $\pi:
\calG\to\pi_0(\calG)$ is a 2-group homomorphism.

Construction in Example \ref{Picard} has an obvious generalization
to give a sheaf of 2-groups over
$(\mathbf{Aff}/S)_{fppf}$\footnote{By abuse of language, we use
$(\mathbf{Aff}/S)_{fppf}$ to either denote the site, or the
corresponding topos.}, usually denoted by $\calP ic_X$, and called
the Picard stack of $X$. It assigns every $U\to S$ the Picard
groupoid $\pic_{X\times_SU}$.

The first objective of this paper is to give a sheaf theoretical
version of the construction in Example \ref{GL(C)} when $\calC$ is
an abelian category.

\section{Quasi-coherent sheaves of abelian categories over schemes
(stacks)}\label{sheaf of abelian cat}

\subsection{Motivation} Let $S$ be a scheme with a given topology (Zariski, \'{e}tale, or flat). Recall that on $S$, there are three successive abelian categories of sheaves,
\[\on{Ab}\supset\on{Mod}(\calO_S)\supset\on{Qcoh}(S).\]
That is, the category of sheaves of abelian categories, the category of $\calO_S$-modules, and the category of quasi-coherent sheaves on $S$ (which does not depend on the choice of the topology listed above).
While a stack of abelian categories over $S$ (w.r.t. the given topology) is the categorical analogue of a sheaf of abelian groups on $S$, it is desire to also have a similar categorical analogue of a quasi-coherent sheaf on $S$. To do this, let us recall that the category of quasi-coherent sheaves on $S$ is equivalent to the category of the following data (we assume $S$ to be separated and quasi-compact for simplicity): for every $f:\spec A\to S$ an $A$-module $M_A$ and for any $\spec B\to \spec A$ over $S$ an isomorphism $M_A\otimes_{A}B\cong M_B$ satisfying the usual compactibility conditions. This characterization of quasi-coherent sheaves on $S$ generalizes well in the categorical setting, which we will explain in what follows.

\subsection{Base change of abelian categories}
We will use the following notation. If $\lambda:F\to G$ is a
morphism between two functors $F$ and $G$ acting from a category
$\calC$ to a category $\calC'$, we denote $\lambda_X:F(X)\to G(X)$
the specialization of $\lambda$ to $X$.

Now let $\calC$ be an abelian category. The center $\calZ(\calC)$
of $\calC$ is by definition the ring $\End\on{Id}_{\calC}$, where
$\on{Id}_{\calC}$ is the identity functor of $\calC$. Thus, an
element $a\in \calZ(\calC)$ assigns to every $X\in\calC$ a
morphism $a_X\in\End_\calC X$ such that for any $f:X\to Y$,
$a_Y\circ f=f\circ a_X$. It is easy to see that $\calZ(\calC)$ is
in fact a commutative ring. For instance, if $\calC=A\Mod$, the
category of left-modules over a ring $A$, then
$\calZ(\calC)=Z(A)$, the center of $A$. For any $X,Y\in\calC$,
$\mbox{Hom}_{\calC}(X,Y)$ is a $\calZ(\calC)$-module.

Let $R$ be a commutative ring. We will say that an abelian
category is $R$-linear, or over $S=\spec R$, if there is a given
map $R\to\calZ(\calC)$. If $\calC$ is $R$-linear, then the
$\Hom_\calC(X,Y)$ will be an $R$-module for any $X,Y\in\calC$.
Observe that every abelian category is over the spectrum of its
center.

\medskip

From now on, we should assume that $\calC$ is an abelian category,
which satisfies (AB5), i.e. it admits small coproducts (and therefore, it admits arbitrary small colimits), and the filtered colimits of exact sequences are exact.

We assume $\calC$ is an abelian category over some base $S=\spec
R$. For any $R\to A$, we will denote $\calC_A$ the category whose
objects are pairs $(X,\varphi)$ where $X\in\calC$ and
$\varphi:A\to \End_{\calC}X$ such that the new $R$-action on $X$
given by $R\to A\stackrel{\varphi}{\to}\End_{\calC}X$ coincides
with the original one, and whose morphisms are those in $\calC$
which are compatible with the action of $A$. It is a routine work
to check that $\calC_A$ is indeed an abelian category. It is clear
that $\calZ(\calC_A)\cong A\otimes_R\calZ(\calC)$. Therefore,
$\calC_A$ is an abelian category over $\spec A$, which is called
the base change of $\calC$ to $A$. Now if $A\to B$ is a ring
homomorphism, there are two abelian categories over $\spec B$. One
is constructed by using $S\to A\to B$ and so the base change of
$\calC$ to $B$. The other is the base change of $\calC_A$ to $B$.
There is a canonical equivalence of categories
$(\calC_A)_B\cong\calC_B$.

Given a homomorphism $f:A\to B$ over $R$, there is the forgetful
functor $f_*:\calC_B\to\calC_A$, which sends $(X,\varphi)$ to
$(X,\varphi\circ f)$. This functor has a left adjoint, whose
construction we recall presently.

There is a well-defined functor of tensor product
\[
A\Mod\times\calC_A\to\calC_A, \ \ \ M,X\mapsto M\otimes_A X.
\]
Namely, for an $A$ module $M$, let $\calM$ be the category with
objects $m\in\calM$ and $\Hom_\calM(m,m')=\{a\in A, m=am'\}$. For
any $X\in\calC_A$, we define the functor
$\calF_X:\calM\to\calC_A$, which sends $(a:m\to m')\in\calM$ to
$a:X\to X$. Then define
\begin{equation}
M\otimes_A X:=\lim\limits_{\longrightarrow}\calF_X.
\end{equation}
It is easy to see this definition coincides with the definition
given in \cite{G}.

Now let $M=B$ be an $A$-algebra, then there is a natural action of
$B$ on the index category $\calB$, and therefore on $B\otimes_A
X$, and the morphism $B\otimes_AX\to B\otimes_AY$ is compatible
with the $B$-structure. This way, one defines a functor
$\calC_A\to\calC_B$ by $f^*(X)=X\otimes_AB$. It is easy to check
that $f^*$ is the left adjoint of $f_*$. $f^*$ is right exact, and
it is exact if $B$ is flat over $A$ (see \cite{G} Lemma 4)\footnote{We sketch the proof here to see why we require that $\calC$ satisfies (AB5). So we will show if $M$ is a flat $A$-module, then $M\otimes_A-$ is exact. If $M$ is projective given by an idenponent of $A^I$, then $M\otimes_AX$ is given by the corresponding idenpotent of $X^I$. Since $X\to X^I$ is exact by (AB4), $M\otimes_A-$ is exact if $M$ is projective. Since every flat $A$-module is a filtered colimit of projective $A$-modules, then by (AB5) $M\otimes_A-$ is exact if $M$ is flat.}.

\begin{rmk}
The pair $(\calC_B,f^*)$ is characterized by the following
universal property. For any $\calD$ an abelian category over
$\spec B$ and $g:\calC_A\to\calD$ an $A$-linear additive functor,
there is a $B$-linear additive functor $\tilde{g}:\calC_B\to\calD$
and a natural transform $\varepsilon:g\cong \tilde{g}\circ f^*$.
Furthermore, such pair $(\tilde{g},\varepsilon)$ is unique up to a
unique isomorphism.
\end{rmk}

\begin{rmk}\label{lax functor}As usual, given $A\stackrel{f}{\to}B\stackrel{g}{\to}C$ over
$S$, there is a canonical isomorphism $g^*\circ f^*\cong(g\circ
f)^*:\calC_A\to\calC_C$ such that the natural compatibility axiom
for 3-fold compositions holds.
\end{rmk}

\begin{ex}\label{K-mod}If $K$ is a (not
necessarily commutative) ring, equipped with a map $R\to Z(K)$,
and $\calC$ is the category of $K$-modules, then for any $f:R\to
A$, $\calC_A\cong (K\otimes_RA)\Mod$. Observe that although
$\calC_A$ is independent of the morphism $f$, the functor $f^*$
depends on $f$.
\end{ex}

\begin{ex}\label{dual numbers}
The following example is instructive. Let $R=k$ be a field, and
$A=D=k[\varepsilon]/(\varepsilon^2)$ be the ring of dual numbers.
Then for any $\calC$ an abelian category over $\spec k$, $\calC_D$
can be described as follows: objects are $(X,d_X)$ where $X$ is an
object of $\calC$ and $d_X\in\mathrm{End}_\calC(X)$ such that
$d_X^2=0$; morphisms are those $(\alpha:X\to Y)\in\calC$ such that
$d_Y\circ \alpha=\alpha\circ d_X$. Let $f:D\to k$ be the ring
homomorphism defined by $f(\varepsilon)=0$. Then
$f^*:\calC_D\to\calC$ is given by
$f^*(X\stackrel{d_X}{\to}X)=\on{coker}d_X$.
\end{ex}

We will need the following simple lemma in \S \ref{gl(C)}.
\begin{lem}\label{closed embedding}
Let $f:A\to B$ be a surjective ring homomorphism, and $\calC$ be an
$A$-linear abelian category. Then the natural adjunction map
$f^*f_*\to \on{Id}$ is an isomorphism
\end{lem}

\subsection{Sheaves of abelian categories} Now, we can follow \cite{G} to make sense the notion of a quasi-coherent sheaf
of abelian categories $\calC$ over a quasi-compact separated
scheme $S$ (even algebraic stacks with affine diagonal) with
respect to flat topology on $\mathbf{Aff}/S$. We first present the
following lemmas, which are the contents of \cite{G}, Proposition
5 (originally, due to Drinfeld) and Proposition 8.

\begin{lem}\label{faithfully flat}
If $f:A\to B$ is a faithfully flat morphism, then the functor
$f^*:\calC_A\to\calC_B$ is exact and faithful.
\end{lem}

For $A\to A'$, set $A''=A'\otimes_AA'$. For any abelian category
$\calC'$ over $A'$, one naturally defines the category of descent
data on $\calC'$ with respect to $A''$, denoted by
$Desc_{A''}(\calC')$. This is an $A$-linear category, equipped
with the forgetful functor $Desc_{A''}(\calC')\to\calC'$. It is in
general not true that $Desc_{A''}(\calC')$ is abelian, but it is
the case if $A'$ is flat over $A$. If $\calC'=\calC_{A'}$ is the
pullback of some $\calC$ over $A$, there is a natural $A$-linear
functor $\calC\to Desc_{A''}(\calC_{A'})$ such that the
composition $\calC\to Desc_{A''}(\calC_{A'})\to\calC_{A'}$ is just
$f^*$. Therefore, if no ambiguity arises, we will denote the
natural functor $\calC\to Desc_{A''}(\calC_{A'})$ also by $f^*$.

\begin{lem}\label{descent}
For any faithfully flat morphism $A\to B$, the natural functor
$\calC_A\to Desc_{B\otimes_AB}(\calC_B)$ is an equivalence of
categories.
\end{lem}

The proof Lemma \ref{descent} is based on Lemma \ref{faithfully
flat} and the usual descent argument.

\begin{dfn}Let $S$ be a quasi-compact and separated scheme. Then a quasi-coherent sheaf of abelian categories $\calC$ over
$S$ is a rule to assign to every $\spec A$ affine over $S$ an
abelian category $\calC_A$ over $A$, and for any $A\to A'$ over
$S$, an equivalence $\calC_{A'}\cong(\calC_A)_{A'}$ satisfying
natural compatibility conditions.\end{dfn}

The following proposition, which is a direct consequence of Lemma \ref{descent}, justifies the sheaf property of $\calC$.

\begin{prop}Let $\calC$ be a quasi-coherent sheaf of abelian categories over $S$. Then the assignment $A\to \calC_A$ and $f:A\to B$, $f^*:\calC_A\to(\calC_A)_B\cong\calC_B$ is naturally a stack of abelian categories on $(\mathbf{Aff}/S)_{fppf}$.
\end{prop}

Let us give some examples of quasi-coherent sheaves of abelian categories over a separated, quasi-compact base $S$.

\begin{ex}\label{aff} Let $S=\spec R$ be affine. Then an $R$-linear abelian category $\calC$ (an abelian category over $S$) gives a quasi-coherent sheaf of
abelian categories over $S$. Namely, it assigns every $R\to A$ the
base change category $\calC_A$. Conversely, given a quasi-coherent
sheaf of abelian categories $\calC$ over $S$, we obtain a single
$R$-linear abelian category from the identity map of $S$. This
assignment will be enhanced to a pair of adjoint functors between
the 2-category of $R$-linear abelian categories and the 2-category
of quasi-coherent sheaves of abelian categories over $\spec R$
later.
\end{ex}

\begin{ex}\label{Qcoh} Let $S$ be a separated quasi-compact scheme and $X$ be
a scheme quasi-compact and separated over $S$. Let $\calQ coh(X)/S$
denote the rule that assigns every $\spec A\to S$ the category of
quasi-coherent sheaves on $X_A:=X\times_S\spec A$. Let us prove that
this is indeed a quasi-coherent sheaf of abelian categories over
$S$. Indeed, choose $U\to X$ be a faithfully flat morphism with $U$
affine. Then $V:=U\times_XU$ is also affine. By the diagram
$V_A\rightrightarrows U_A\to X_A$, $\on{Qcoh}(X_A)$ is canonically
equivalent to the category $Desc_{V_A}(\on{Qcoh}(U_A))$ of
quasi-coherent sheaves on $U_A$ equipped with the descent data with
respect to $V_A$, i.e. category of $\Gamma(U_A,\calO_{U_A})$-modules
with additional data. Since the base change of
$\Gamma(U_A,\calO_{U_A})\Mod$ from along $A\to B$ is just
$\Gamma(U_B,\calO_{U_B})\Mod$ (see Example \ref{K-mod}), it is not
hard to see that the base change of $Desc_{V_A}(\on{Qcoh}(U_A))$ for
any $A\to B$ is canonically equivalent to
$Desc_{V_B}(\on{Qcoh}(U_B))$. Therefore, $\on{Qcoh}(X_A)_B$ is
canonically equivalent to $\on{Qcoh}(X_B)$.

Clearly, if $S=\spec R$, then the single $R$-linear abelian category
corresponding to the sheaf of abelian categories $\calQ coh(X)/S$
over $S$ is just $\on{Qcoh}(X)$.
\end{ex}

\begin{ex}\label{twist}The above example has a direct generalization to twisted sheaves.
Let $X\to S$ be as above, and let $\calG$ be a $\bbG_m$-gerbe over
$X$, determined by a cohomology class $\alpha\in
H^2_{\on{fppf}}(X,\bbG_m)$. Then it makes sense to talk about the
category of $\calG$-twisted quasi-coherent sheaves on $X$ (cf.
\cite{L}), denoted by $\on{Qcoh}_{\calG}(X)$. The sheaf version then
is to assign every $\spec A\to S$ the category
$\on{Qcoh}_{\calG_A}(X_A)$. By the same reason as above, this is a
quasi-coherent sheaf of abelian categories over $S$, denoted by
$\calQ coh_\calG(X)$.
\end{ex}

We can generalize the notion of the center of a single abelian
category to a quasi-coherent sheaf of abelian categories. Let $\calC$ be a quasi-coherent sheaf
of abelian categories over $S$. Then we define $\calZ(\calC)$ as a
sheaf over $S$ which assigns every $\spec A\to S$ the set of $\spec
A$-morphisms $\spec A\to \spec\calZ(\calC_A)$. By the theory of
descent, $\calZ(\calC)$ is represented by a scheme affine over
$S$. Clearly, if $S$ is affine, the new notion coincides with the
old one.

If $\calC$ is a quasi-coherent sheaf of abelian categories over $S$, it makes
sense to talk about the category of global sections of $\calC$,
which is a single abelian category, denoted by $\Gamma(S,\calC)$.
Namely, for any such a scheme $S$, let $U\to S$ be a covering of
$S$ by affine schemes. Then $U\times _YU$ is also affine. Now an
abelian category over $X$ is defined as the
$Desc_{U\times_TU}(\calC_U)$. Lemma \ref{descent} guarantees the
above definition is valid. For example,
\[\Gamma(S,\calQ coh(X)/S)\cong\on{Qcoh}(X)\]
regardless $S$ is affine or not.

\section{The 2-group of auto-equivalences of an abelian
category}\label{2-grp of auto-equiv}
\subsection{The 2-group $\GL(\calC)$}
Now if $\calC$ is a quasi-coherent sheaf of
abelian categories over $S$, it is straight forward to generalize
Example \ref{GL(C)} to associate a sheaf of 2-groups $\GL(\calC)$.

For any $\spec A\to S$, We denote by $\GL_A(\calC_A)$ the 2-group
whose objects are $A$-linear auto-equivalences of $\calC_A$ and
morphisms are isomorphisms between these auto-equivalences. The
following is the main result of this section.

\begin{thm}\label{sheaf of auto-equivalences}
There is a sheaf of   2-groups $\GL(\calC)$ over
$(\mathbf{Aff}/S)_{fppf}$, which assigns to $\spec A\to S$ the
2-group $\GL_A(\calC_A)$.
\end{thm}

\begin{proof}
First, we should realize $\GL(\calC)$ as a category fibered in
groupoids over $\mathbf{Aff}/S$. Therefore we should define, for
any $f:A\to B$ over $S$, a pullback functor
$f^*:\GL_A(\calC_A)\to\GL_B(\calC_B)$, satisfying the usual axioms
for composition.

Recall that objects in $(\calC_A)_B$ are pairs $(X,\varphi)$ where
$X\in\calC_A$ and $\varphi:B\to\End_{\calC_A}(X)$ such that
$\varphi\circ f$ coincides with the original $A$-structure on
$\End_{\calC_A}(X)$, and morphisms $u:(X,\varphi)\to (X',\varphi')$
in $\calC_B$ are those morphisms in $\calC_A$ which commute with the
action of $B$. Therefore, it is natural to define, for
$F\in\GL_A(\calC_A)$, the object $f^*(F)$ of $\GL_B((\calC_A)_B)$ by
the formula
\[
f^*(F)((X,\varphi)\stackrel{u}{\to}(X',\varphi'))=(F(X),F(\varphi))\stackrel{F(u)}{\to}(F(X'),F(\varphi')),
\]
where $F(\varphi)$ is
the map $B\to\End(X)\stackrel{F}{\to}\End(F(X))$. It is clear that
$f^*(F)$ is a well-defined object in $\GL_B((\calC_A)_B)$.

Next, we define, for a morphism $a:F\to G$ in $\GL_A(\calC_A)$,
(i.e., a collection of isomorphisms $a_X:F(X)\to G(X)$ for all
$X\in \calC_A$, functorial in $X$), a morphism $f^*a:f^*(F)\to f^*(G)$
in $\GL_B((\calC_A)_B)$. Observe that for any
$(X,\varphi)\in\calC_B$, we always have $a_X\circ
F(\varphi(b))=G(\varphi(b))\circ a_X$ for $b\in B$. Therefore,
\[a_X\in\Hom_{\calC_B}((F(X),F(\varphi)),(G(X),G(\varphi))).\] We
define $f^*a:f^*(F)\to f^*(G)$ by assigning to $(X,\varphi)$ the
morphism $a_X$. It is clear that this assignment is functorial in
$(X,\varphi)$, and therefore defines a morphism in
$\GL_B((\calC_A)_B)$.

Let us summarize. We have defined for any $f:A\to B$, a pullback
functor $f^*:\GL_A(\calC_A)\to\GL_B((\calC_A)_B)$. Recall that the
data of a quasi-coherent sheaf of abelian categories over $S$ contains a
canonical equivalence $\calC_B\cong(\calC_A)_B$. Let us choose
once and for all a quasi-inverse of this equivalence for any
$S$-morphism $f: A\to B$. Then we obtain an equivalence
$\GL_B((\calC_A)_B)\cong\GL_B(\calC_B)$. We still denote the
composition
$\GL_A(\calC_A)\to\GL_B((\calC_A)_B)\cong\GL_B(\calC_B)$ by $f^*$.
It is clear that $f^*$ is a monoidal functor, and since the
equivalence $\calC_B\cong(\calC_A)_B$ satisfies the natural
compatibility conditions, there is a canonical isomorphism
$g^*\circ f^*\cong(g\circ f)^*$ such that the natural
compatibility condition for 3-fold compositions holds. Therefore,
$\GL(\calC)$ is indeed a category fibered in groupoids over
$\mathbf{Aff}/S$.

Next, we need the following lemma.

\begin{lem}\label{f^*F=(f^*F)f^*}
For any $F\in\GL_A(\calC_A)$, there   is a canonical isomorphism
$f^*\circ F\cong f^*(F)\circ f^*:\calC_A\to\calC_B$.
\end{lem}

\begin{proof}
Clearly, we could assume that $\calC_B=(\calC_A)_B$. We recall the
definition of $f_*(f^*(X))$. Let $\calB$ be the category with objects
$b\in B$ and $\Hom_\calB(b,b')=\{a\in A,b=ab'\}$, and
$\calF_X:\calB\to\calC_A$ be the functor which sends $(a:b\to
b')\in\calB$ to $(a:X\to X)\in\calC_A$. Then
$f_*(f^*(X))=\lim\limits_{\longrightarrow}\calF_X$ is the colimit of
the functor $\calF_X$. Observe that if $H:\calC_A\to\calC_A$ is
any $A$-linear additive functor, there is a canonical morphism
$\lim\limits_{\longrightarrow}\calF_{H(X)}\to
H(\lim\limits_{\longrightarrow}\calF_X)$ by the universal property
of colimits. Furthermore, if $H$ is an auto-equivalence, then this
morphism is an isomorphism. Therefore, we obtain a canonical
isomorphism $f_*(f^*(F(X)))\cong F(f_*(f^*(X)))$ for any
$F\in\GL_A(\calC_A)$ and $X\in\calC_A$, which is functorial in
$X$. It is clear that the above isomorphism commutes with the
action of $B$ on both sides, and this gives us the required
canonical isomorphism $f^*(F(X))\cong f^*(F)(f^*(X))$, functorial in
$X$.
\end{proof}

Now we show that $\GL(\calC)$ is in fact a stack with respect to
the flat topology on $\mathbf{Aff}/S$.

For $\spec A\to S$, denote $s_A:\spec\calZ(\calC_A)\to\spec A$,
and let $I_A$ be the unit of $\GL_A(\calC_A)$. Then for any
$f:A\to B$ over $S$, $\End
I_B=\Hom_S(\spec\calZ(\calC_B),\bbG_m)\cong\Hom_S(\spec
B\times_{\spec A}\spec\calZ(\calC_A),\bbG_m)$. Therefore, the
functor $(A\to B)\mapsto \End I_B$ is represented by
$(s_A)_*\bbG_m$ and hence it is a sheaf. If
$F,G\in\GL_A(\calC_A)$, then $\Hom(F,G)$ is an $\End I_A$-torsor,
and therefore it is also a sheaf.

Let $f:A\to B$ be a faithfully flat morphism, and $i_1,i_2:B\to
B\otimes_AB$. Let $F:\calC_B\to\calC_B$ be a $B$-linear
auto-equivalence with an isomorphism $u:i_1^*(F)\cong i_2^*(F)$ such
that the cocycle condition over $B\otimes_AB\otimes_AB$ is
satisfied. We need to show that there exists
$F_0\in\GL_A(\calC_A)$ and an isomorphism $u_0:f^*(F_0)\cong F$ such
that $u$ is the composition of the natural isomorphisms
$i_1^*(F)\cong i_1^*(f^*(F_0))\cong i_2^*(f^*(F_0))\cong i_2^*(F)$.

Without loss of generality, we could assume that
$\calC_B=(\calC_A)_B$,
$\calC_{B\otimes_AB}=(\calC_A)_{B\otimes_AB}$, etc. Recall that
since $B$ is flat over $A$, $Desc_{B\otimes_AB}(\calC_B)$ is an
$A$-linear abelian category. We claim that the datum $(F,u)$ defines
an object in $\GL_A(Desc_{B\otimes_AB}(\calC_B))$. Indeed, let
$(X,v)\in Desc_{B\otimes_AB}(\calC_B)$, where $X\in\calC_B$ and
$v:i_1^*(X)\cong i_2^*(X)$ satisfying the usual cocycle condition
over $B\otimes_AB\otimes_AB$. We claim that $F(X)$ is naturally an
object in $Desc_{B\otimes_AB}(\calC_B)$. Indeed, we define the
isomorphism $F(v):i_1^*(F(X))\cong i_2^*(F(X))$ as the composition
\[
i_1^*(F(X))\cong i_1^*(F)(i_1^*(X))\cong i_2^*(F)(i_2^*(X))\cong i_2^*(F(X)),
\]
where the first and the last isomorphisms are due to
Lemma \ref{f^*F=(f^*F)f^*}, and the middle one is due to
$u:i_1^*(F)\cong i_2^*(F)$ and $v:i_1^*(X)\cong i_2^*(X)$. Since $u$ and
$v$ satisfy the cocycle condition over $B\otimes_AB\otimes_AB$, so
does $F(v)$. It is clear that the construction is functorial in
$(X,v)$ and is $A$-linear. Therefore, $(F,u)$ gives rise to an
object in $\GL_A(Desc_{B\otimes_AB}(\calC_B))$. By Lemma
\ref{descent}, we choose an equivalence
$E:Desc_{B\otimes_AB}(\calC_B)\to\calC_A$ which is quasi-inverse
to $f^*:\calC_A\to Desc_{B\otimes_AB}(\calC_B)$. Then it gives
rise to an isomorphism of 2-groups
$\GL_A(Desc_{B\otimes_AB}(\calC_B))\to\GL_A(\calC_A)$ We define
$F_0:\calC_A\to\calC_A$ to be the image of $(F,u)$ under the
morphism. It is routine work to check that $F_0$ satisfies the
required properties. This finishes the proof of Theorem \ref{sheaf
of auto-equivalences}. \end{proof}

The key point of the above theorem is Lemma \ref{f^*F=(f^*F)f^*}.
This lemma could be reformulated as below by saying that there is a functor
from the 2-category of $R$-linear abelian categories to the
2-category of quasi-coherent sheaves of abelian categories over $S=\spec R$. This is the categorical analogue of the usual localization functor from the category of $R$-modules to the category of quasi-coherent sheaves on $\spec R$.

Let us define the 2-category of quasi-coherent sheaves of abelian categories over
$S$.
\begin{dfn}
For $\calC$, $\calD$ two quasi-coherent sheaves of abelian categories over
$S$, a 1-morphism $F$ from $\calC$ to $\calD$ amounts to assign
for every $\spec A\to S$, and $A$-linear additive functor
$F_A:\calC_A\to\calD_A$, and for every $f:A\to B$ and isomorphism
$\phi_{A,B}:F_B\circ f^*\cong f^*\circ F_A$ satisfying the usual compatibility
conditions. A 2-morphism from $F$ to $G$ is a natural transform
that is compatible with those $\phi_{A,B}$'s.
\end{dfn}

Now let $S=\spec R$ be affine, and $\calC,\calD$ are two
$R$-linear abelian categories. By Example \ref{aff}, $\calC,\calD$
give sheaves of abelian categories over $S$, denoted by
$\tilde{\calC},\tilde{\calD}$. Now if $F:\calC\to\calD$ is an
$R$-linear additive functor, then by Lemma \ref{f^*F=(f^*F)f^*},
for every $f:R\to A$, $f^*(F)$ is an $A$-linear additive functor
from $\tilde{\calC}_A$ to $\tilde{\calD}_A$ and there is a
canonical isomorphism $f^*(F)\circ f^*\cong f^*\circ F$. Therefore,
$F$ gives rise to a 1-morphism $\tilde{F}$ from $\tilde{\calC}$ to
$\tilde{\calD}$ by demanding $\tilde{F}_A=f^*(F)$ for any $f:R\to A$. If we work carefully, we will really obtain a
functor from the 2-category of $R$-linear abelian categories to
the 2-category of sheaves of abelian categories over $S=\spec R$, which is called the localization functor.

One would expect that this localization functor would induce an equivalence
between the 2-category of $R$-linear abelian categories and the
2-category of quasi-coherent sheaves of abelian categories over $S=\spec R$, as the usual localization functor does.
However, it is not the case. The reason is that not every
1-morphism $\tilde{F}:\tilde{C}\to\tilde{D}$ between quasi-coherent sheaves of abelian categories $\tilde{\calC},\tilde{\calD}$ over $\spec R$ satisfies
$\tilde{F}_A\cong f^*(\tilde{F}_R)$ for any $f:R\to A$. The best one
can prove (for example, by using Lemma \ref{res}) is that if
$\tilde{F}_A$ preserves colimits, then $\tilde{F}_A\cong
f^*\tilde{F}_R$. This suggests that in the definitions of the 2-category of $R$-linear abelian categories and the 2-category of quasi-coherent sheaves of abelian categories over $\spec R$, one should only allow 1-morphisms to be those colimit preserving (i.e. right exact) additive functors. If we modify our definitions in this way, we will obtain that the localization functor is a 2-equivalence.

\subsection{Part of the center of $\GL(\calC)$}

Recall that if $V$ is a (finite dimensional) vector space over a field, then the center of $GL(V)$
is $\bbG_m$. We would like to give a partial analogous statement
for $\GL(\calC)$. To this end, let us first review the central
functor from a symmetrical monoidal category to a monoidal
category. Let $\calC$ be a symmetrical monoidal category in this
paragraph (as opposed to elsewhere of the paper where $\calC$
usually denotes an abelian category), and $\calD$ be a monoidal
category. A monoidal functor $\calZ:\calC\to\calD$ is called
central if for any $x\in\calC, y\in\calD$, there is a functoral
isomorphism
\[
\sigma_{x,y}:\calZ(x)y\cong y\calZ(x)
\]
satisfying all the usual compatibility conditions.
\begin{prop}Let
$\calC$ be a sheaf of abelian categories over $S$. Then there is
an embedding $\calZ:\calP ic_S\to\GL(\calC)$. Furthermore, this
embedding is central. \end{prop}

\begin{proof}
Recall that we have an action of $\calQ coh(S)$ on $\calC$, i.e.,
for every $\spec A\to S$, there is a bifunctor $\otimes:
A\Mod\times\calC_A\to\calC_A$. Remark that if $f:A\to B$, then
there is a natural isomorphism between $f^*(-\otimes-)$ and
$f^*(-)\otimes f^*(-)$. We define a monoidal functor $\calZ_A:
\pic_{\spec A}\to\GL_A(\calC_A)$ by $L\mapsto L\otimes-$. We show
that this functor is central.

Let $F\in\GL_A(\calC_A)$, and $L\in\mathrm{Pic}_A$. Let us show
that there is a functorial isomorphism $F(L\otimes X)\cong
L\otimes F(X)$ for any $X\in\calC_A$. Let $f:A\to B$ be faithfully
flat such that $L\otimes_AB\cong B$. We fix such an isomorphism.
Then, by Lemma \ref{f^*F=(f^*F)f^*}, we obtain a functorial
isomorphism
\[\begin{split}
f^*(F(L\otimes X))&\cong f^*(F)(f^*(L)\otimes f^*(X))\cong f^*(F)(f^*(X))\\
                               &\cong f^*(L)\otimes f^*(F(X))\cong f^*(L\otimes F(X)).
\end{split}\]
It is readily to check that this isomorphism is compatible with
the descent data, and therefore gives a functorial isomorphism
$F(L\otimes X)\cong L\otimes F(X)$.

Finally, all the local constructions behave well under pullbacks
and the global statement follows. \end{proof}

Let $\calZ(\calC)$ be the center of $\calC$. Recall that this is a
scheme affine over $S$. It is clear that there is an embedding
$\calP ic_{\calZ(\calC)}\to\GL(\calC)$. However, this embedding is
not central in general.

\section{Example: a description of $\GL(\calQ coh(X))$}\label{example}
Let us give an example of $\GL(\calC)$. Let $X\to S$ be a scheme
separated and quasi-compact over a separated, quasi-compact scheme
$S$. Recall from Example \ref{Qcoh} that $\calC=\calQ coh(X)$ is
the sheaf of abelian categories over $S$ which assigns every
$\spec A\to S$ the category of quasi-coherent sheaves on
$X_A:=X\times_S\spec A$.

It is clear that there is an embedding $\calP ic_X\to\GL(\calQ
coh(X))$ which sends an invertible sheaf $\calL$ on $X_A$ to the
auto-equivalence $-\otimes\calL$. On the other hand, the Aut sheaf
$\underline{\on{Aut}}_S(X)$ also embeds in $\GL(\calQ coh(X))$.
The main theorem of this section is

\begin{thm}\label{Aut}Let $X$ be a quasi-compact and separated over $S$. Then there is a
splitting short exact sequence
\[1\to\calP ic_X\to\GL(\calQ coh(X))\to\underline{\on{Aut}}_S(X)\to 1\]
\end{thm}

The theorem is a direct consequence of the following Proposition
\ref{Auto}. However, if $X=S$, the proof is much simpler.

\begin{prop}\label{Picard stack as GL(C)}
$\GL(\calQ coh(S))\cong \calP ic_S=[S/\bbG_m]$, the Picard stack
of $S$.
\end{prop}
\begin{proof}
It is clear that we only need to prove that $\GL_A(\on{Qcoh}(\spec
A)) \cong \pic_{\spec A}$. Let us denote $\on{Qcoh}(\spec A)$ by
$A\Mod$.

First, we claim that the functor $F:A\Mod\to A\Mod$ is isomorphic
to \[-\otimes F(A):A\Mod\to A\Mod.\] Indeed, we present any
$A$-module $M$ by \[A^I\stackrel{\varphi}{\to} A^J\to M\to 0.\]
Since $F$ is an $A$-linear auto-equivalence, we have the following
right exact sequences
\[\begin{CD} F(A)^I@>F(\varphi)>> F(A)^J@>>> F(M)@>>>0\\ @V\cong VV@V\cong
VV@.\\ A^I\otimes F(A)@>\varphi\otimes F(A)>>A^J\otimes F(A)@>>>
M\otimes F(A)@>>>0.
\end{CD}\]
Therefore, $F(M)\cong M\otimes F(A)$, functorially in $M$.

Next, let $G$ be a quasi-inverse of $F$. We have $A\cong
G(F(A))\cong F(A)\otimes G(A)$. Therefore, $F(A)$ is an invertible
$A$-module.

Now, the functor $\psi:\GL_A(A\Mod)\to \pic_{\spec A}$ is given as
follows: for $F\in\GL_A(A\Mod)$, $\psi(F)=F(A)$. Its
quasi-inverse $\varphi:\pic_{\spec A}\to\GL_A(A\Mod)$ is
given by $\varphi(L)=-\otimes L$. \end{proof}

Now we turn to the general case.

\begin{prop}\label{kernel}Let $S=\spec R$ and $X,Y$ are two $S$-schemes, with
$X$ separated and quasi-compact over $S$. Let
$F:\on{Qcoh}(X)\to\on{Qcoh}(Y)$ be an exact $R$-linear functor.
Then there is a unique (up to isomorphism) quasi-coherent sheaf
$\calK$ on $X\times_SY$, such that $F$ is isomorphic to the
functor $\Phi_{\calK}(-)=q_*(p^*(-)\otimes \calK)$, where
$p:X\times_SY\to X, q:X\times_SY\to Y$ are two projections.
\end{prop}
\begin{rmk}Several remarks are in order. First, I think the proposition should hold for right exact
functors $F$, probably under certain restrictions\footnote{For
example, if $F$ extends to an exact DG-functor, then one can use
\cite{T}.}, although the proof presented as below does not apply to
this stronger statement. But if the functor $F$ is exact, then as it
will be clear from the proof, $\calK$ is flat over $X$.

Second, in literature, there exist much deeper theorems concerning
about the exact functors between derived categories (or
DG-categories) of quasi-coherent sheaves (cf. \cite{O} Theorem
2.2, and \cite{T} Theorem 8.9), from which this proposition can be
deduced (but under further restrictions on $X$ and $Y$). This
point here is that the statement of the proposition is general and
the proof is elementary.

Third, this proposition has a twisted version. Namely, if $F$ is an
$R$-linear exact functor from $\on{Qcoh}_\calG(X)$ to
$\on{Qcoh}_\calH(Y)$ (cf. Example \ref{twist}), where $\calG$ (resp.
$\calH$) is a $\bbG_m$-gerbe on $X$ (resp. $Y$), then there exists a
unique (up to isomorphism) $\calG^{-1}\boxtimes\calH$-twisted
quasi-coherent sheaf $\calK$ on $X\times_S Y$ such that $F$ is
isomorphic to $q_*(p^*(-)\otimes\calK)$. For the proof, one just
need to replace the Zariski open covers $U$ of $X$ (resp. $V$ of
$Y$) in below by \'{e}tale covers such that the the pullback of
$\calG$ (resp. $\calH$) to $U$ (resp. $V$) is trivial.
\end{rmk}
\begin{proof}
\noindent\emph{Step I}. Assume that both $X=\spec A$ and $Y=\spec
B$ are affine and $F$ is right exact. Let $K=F(A)$, this is a
$B$-module. From
\[A\cong\End_{A-\on{Mod}}(A)\stackrel{F}{\to}\End_{B-\on{Mod}}(F(A)),\]
$K$ obtains an $A$-module structure, and since the functor $F$ is
$R$-linear, $K$ is in fact an $(A\otimes_RB)$-module. It is clear
that if $M=A^I$ is free, then $F(M)\cong\Phi_K(M)$. Then using the
right exactness of $F$ and the same argument as in the proof of
Proposition \ref{Picard stack as GL(C)}, we conclude that
$F\cong\Phi_K$. The uniqueness is clear.

In this way, we obtain a functor from the category of $R$-linear
right exact functors from $A\Mod$ to $B\Mod$ to the category of
$(A\otimes_RB)$-modules. In what follows, given such an $R$-linear
right exact functor $F:A\Mod\to B\Mod$, the corresponding
$(A\otimes_RB)$-module $K$ is sometimes denoted by $\Psi_F$. It is
clear that if $F$ is exact, then $K$ is flat over $A$.

Let us also remark that this construction localizes well at the
sources and targets. Namely, if $i:U\to X$ and $j:V\to Y$ are an
affine open subschemes of $X$ and $Y$ respectively, then we have a
canonical isomorphism $\Psi_F|_{U\times_SV}\cong\Psi_{j^*\circ
F\circ i_*}$. (Observe that $j^*$ is always exact and since $i$ is
affine, $i_*$ is also exact.)

\medskip

\noindent\emph{Step II}. Assume that $X$ is affine, $Y$ is
arbitrary and $F$ is right exact. Let $Y=\cup_iY_i$ where
$Y_i=\spec B_i$ are affine open subschemes of $Y$. Let
$j^*_i:\on{Qcoh}(Y)\to\on{Qcoh}(Y_i)$ be the restrictions. Then
$j^*\circ F$ is still right exact and $R$-linear. We thus obtain
quasi-coherent sheaves $\calK_i$ on $X\times_SY_i$. By the remark
at the end of the previous step, they glue together to get a
quasi-coherent sheaf $\calK$ on $X\times_SY$, independent of the
choice of the affine open over of $Y$. Let us prove that
$F\cong\phi_\calK$. But this follows from $j_i^*\circ
F\cong\Phi_{\calK_i}\cong j_i^*\circ\Phi_\calK$. Again, we will
denote $\Psi_F$ the quasi-coherent sheaf $\calK$ on $X\times_SY$.
By the remark at the end of the previous step, this construction
localizes well at the source. Namely, if $i:U\to X$ is an affine
open subscheme, then there is a canonical isomorphism
$\Psi_F|_{U\times_SY}\cong\Psi_{F\circ i_*}$.

\medskip

\noindent\emph{Step III}. Let us first assume that $X$ is
quasi-compact and separated over $S=\spec R$, $Y$ is arbitrary,
and $F$ is right exact. Let $X=\cup_iX_i$ where $X_i=\spec A_i$
are affine open subschemes of $X$. Let $j_i:X_i\to X$ be the open
immersion. They are affine morphisms. Then $F\circ(j_i)_*$ is
right exact and $R$-linear. By the previous step, for each $i$,
one constructs a quasi-coherent sheaf $\calK_i$ on $X_i\times_S Y$
such that $F\circ (j_i)_*\cong\Phi_{\calK_i}$. By the reasoning at
the end of the previous step, we can glue these $\calK_i$ together
to get a quasi-coherent sheaf $\calK$ on $X\times_SY$, independent
of the choice of the affine cover of $X$. From the construction,
we know that if $i:U\to X$ is an affine open subscheme, then
$F\circ i_*\cong \Phi_{\calK|_{U\times_SY}}$.

Now we prove that $F\cong\Phi_\calK$ under the further assumption
that $F$ is exact. Clearly, we need only to prove for each affine
open immersion $j:V\to Y$, there is a canonical isomorphism
$j^*\circ F\cong j^*\circ\Phi_\calK$ and such isomorphism respect to
the localization. Therefore, without loss of generality, we can
assume that $Y$ is affine.

Since $X$ is quasi-compact, we can find a finite affine open cover
$X=\cup_iX_i$. Denote $j_{i_1\ldots i_k}:X_{i_1\ldots i_k}:=X_{i_1}\cap\cdots\cap X_{i_k}\to X$
be the open immersion. Then we have the \v{C}ech resolution of the
identity
\[
\on{id}\to(\bigoplus_{i}(j_i)_*j_i^*\to\bigoplus_{i_1i_2}(j_{i_1i_2})_*j_{i_1i_2}^*\to\cdots).
\]
Since $F$ is exact, we obtain the left exact sequence
\[
0\to F\to \bigoplus_i F\circ(j_i)_*j_i^*\to\bigoplus_{i_1i_2}F\circ(j_{i_1i_2})_*j_{i_1i_2}^*.
\]
Or
\[
0\to F\to \bigoplus_i \Phi_{\calK|_{X_i\times_SY}}j_i^*\to\bigoplus_{i_1i_2}\Phi_{\calK|_{X_{i_1i_2}\times_SY}}j_{i_1i_2}^*.
\]
On the other hand, observe that since $Y=\spec C$ is affine
$\Phi_\calK(\calF)$ is the sheaf associated to the $C$-module
$\Gamma(X\times_SY,p^*\calF\otimes\calK)$. The usual \v{C}ech
complex tells us we have the left exact sequence
\[
0\to\Gamma(X\times_SY,p^*\calF\otimes\calK)\to\bigoplus_i(p^*\calF\otimes\calK)(X_i\times_SY)\to\bigoplus_{i_1i_2}(p^*\calF\otimes\calK)(X_{i_1i_2}\times_SY).
\]
Observe that the quasi-coherent sheaf on $Y$ corresponding to the
$C$-module $(p^*\calF\otimes\calK)(X_i\times_SY)$ is
$\Phi_{\calK|_{X_i\times_SY}}j_i^*(\calF)$, and the one corresponding to
$(p^*\calF\otimes\calK)(X_{i_1i_2}\times_SY)$ is $\Phi_{\calK|_{X_{i_1i_2}\times_SY}}j_{i_1i_2}^*(\calF)$. Therefore, we
obtain that there is a canonical isomorphism
\[F(\calF)\cong\Phi_{\calK}(\calF)\]
for any quasi-coherent sheaf $\calF$ on $X$. Or
$F\cong\Phi_{\calK}$.
\end{proof}

If in addition, $F$ is an equivalence, then the sheaf $\calK$ has
a very special form.

\begin{prop}\label{Auto}Assumptions are as in Proposition \ref{kernel}. If in
addition, $Y$ is also separated and quasi-compact over $S$, and
the functor $F$ is an $R$-linear equivalence, then the sheaf
$\calK$ as in Proposition \ref{kernel} is of the following form:
there exists a unique $S$-scheme isomorphism $f:X\to Y$, and a
unique up to isomorphism invertible sheaf $\calL$ on $X$ such that
$\calK$ is isomorphic $(p|_{\Gamma_f})^*\calL$, where $\Gamma_f$
is the graph of $f$, regarded as a closed subscheme of
$X\times_SY$.
\end{prop}
\begin{rmk}In \cite{Ga}, Gabriel proved that if $X,Y$ are two noetherian schemes, and $\on{Qcoh}(X)\cong\on{Qcoh}(Y)$, then $X\cong Y$. His proof does not apply here since it uses noetherian induction. In addition, I am not clear whether his argument can deduce the above statement even in the noetherian case.
\end{rmk}

\begin{proof}Let us begin with three general observations.

First, $F$ necessarily sends locally finitely presented
quasi-coherent sheaves to locally finitely presentable
quasi-coherent sheaves, since these sheaves are compact objects in
$\on{Qcoh}(X)$ and in $\on{Qcoh}(Y)$ (i.e., quasi-coherent sheaves
such that $\Hom(\calF,-)$ commutes with filtered colimits.)

Second, if $\calK$ is the kernel on $X\times_SY$ representing $F$,
i.e. $F\cong\Phi_\calK$, then $\calK$ is locally finitely generated
on $X\times_SY$. The reason is as follows. One can write
$\calK=\lim\calK_\lambda$ as a filtered colimit, with each
$\calK_\lambda$ being a locally finitely generated subsheaf of
$\calK$. Then
\[F\cong q_*(p^*(-)\otimes\lim\calK_\lambda)\cong
\lim\Phi_{\calK_\lambda}\] in the category
$\Hom_R(\on{Qcoh}(X),\on{Qcoh}(Y))$ of $R$-linear functors from
$\on{Qcoh}(X)$ to $\on{Qcoh}(Y)$. We claim that in
$\Hom_R(\on{Qcoh}(X),\on{Qcoh}(Y))$,
\[\Hom(F,F)\cong\lim\Hom(F,\Phi_{\calK_\lambda}).\] Indeed, the
category $\Hom_R(\on{Qcoh}(X),\on{Qcoh}(Y))$ is canonically
equivalent to the category
$\Hom_R(\on{Qcoh}^{lfp}(X),\on{Qcoh}(Y))$, where
$\on{Qcoh}^{lfp}(X)$ is the category of locally finitely presentable
quasi-coherent sheaves on $X$. But by the first observation,
$\Hom(F,F)\cong\lim\Hom(F,\Phi_{\calK_\lambda})$ in
$\Hom_R(\on{Qcoh}^{lfp}(X),\on{Qcoh}(Y))$. The claim follows. Now
the identity morphism in $\Hom(F,F)$ gives an isomorphism
$F\cong\Phi_{\calK_\lambda}$ for some $\lambda$. By the uniqueness
of $\calK$, we have $\calK\cong\calK_\lambda$.

The final observation is that $F$ will map irreducible objects in
$\on{Qcoh}(X)$ bijectively to irreducible objects in $\on{Qcoh}(Y)$
and the irreducible objects of $\on{Qcoh}(X)$ are of the form
$\kappa(x)$, where $x$ is a closed point on $X$ and $\kappa(x)$ the
residue field of $x$, regarded as a skyscraper sheaf supported at
$x$. Let us also recall that on a quasi-compact scheme, there always
exist closed points.

Let us make the following convention. If $i:Z\to X$ is a closed
subscheme, then we say a quasi-coherent $\calO_X$-module $\calF$
is a quasi-coherent sheaf on $Z$ if there is a quasi-coherent
$\calO_Z$-module $\calF'$ such that $\calF\cong i_*\calF'$.

\medskip

\noindent\emph{Special Case}. We first consider the case that $Y$
is affine over $S$. Assume that $Y=\spec B$. Let
$A=\End(\calO_X)$, and $X_a=\spec A$. From the map
\[
\varphi:B\to\End(F(\calO_X))\cong\End(\calO_X)=A,
\]
we obtain a morphism $X_a\to Y$. Let us denote the composition
$f:X\to X_a\to Y$. We claim that the sheaf $\calK$ is a
quasi-coherent sheaf on $\Gamma_f$. Let us cover $X$ by affine
open subschemes $X=\cup_iX_i$, where $j_i:X_i=\spec A_i\to X$ are
open immersions. We have the natural $R$-algebra homomorphism
$r_i:A\to A_i$. Then the map $f\circ j_i:X_i\to Y$ is induced from
the $R$-algebra homomorphism
$B\stackrel{\varphi}{\to}A\stackrel{r_i}{\to} A_i$. It is easy to
see that this homomorphism is the same as
\[
\phi_i:B\to\End(F(j_i)_*\calO_{X_i})\cong\End((j_i)_*\calO_{X_i})=A_i.
\]
Clearly, $\calK|_{X_i\times_SY}$ as a $(A_i\otimes_RB)$-module is
annihilated by $\{1\otimes b-\phi_i(b)\otimes 1=1\otimes
b-r_i(\varphi(b))\otimes 1, b\in B\}$. That is,
$\calK|_{X_i\times_SY}$ is a quasi-coherent sheaf on $\Gamma_{f\circ
j_i}$. Therefore, $\calK$ is a quasi-coherent sheaf on $\Gamma_f$.
The projection $p:X\times_SY\to X$ induces an isomorphism
$\Gamma_f\cong X$. Therefore, $\calK$ can be regarded as a
quasi-coherent sheaf on $X$ and the functor $F$ is isomorphic to
$f_*(-\otimes\calK)$. By the second observation made at the
beginning of the proof, $\calK$ is locally finitely generated on
$X$. Then since $\calK$ is $X$-flat, it is a locally free sheaf of
finite rank. Since it sends $\kappa(x)$ to $\kappa(f(x))$, it is
invertible. It remains to prove that $f$ is an isomorphism. As the
functor $f_*\cong F\circ(-\otimes\calK^{-1})$ is an equivalence,
$f_*$ is exact (as functors from $\on{Qcoh}(X)$ to $\on{Qcoh}(Y)$).
From this we conclude that $X$ is affine since $Y$ is affine. Now
apply the same discussion to a quasi-inverse $F^{-1}$, we obtain a
morphism $g:Y\to X$ and an isomorphism of functors
$g_*\cong(-\otimes\calK)\circ F^{-1}$. From there, it is easy to see
that $f$ is an isomorphism with the inverse $g$.

\medskip

\noindent\emph{General case}. First we have the following observation. Let $X$ be a scheme and $U$ be an open subscheme of $X$. Then the natural embedding $i:U\to X$ is quasi-compact if and only if the closed complement $Z=X-U$ can be equipped with some closed subscheme struture of $X$ such that $\calO_Z$ is locally of finite presentation as an $\calO_X$-module.

Now we return our case where $X,Y$ are separated and quasi-compact. Let $i:U\to X$ be a quasi-compact open embedding. By the above observation, we can give $Z=X-U$ a closed subscheme structure of $X$ such that $\calO_Z$ is locally of finite presentation as an $\calO_X$. Then $F(\calO_Z)$ is a locally finitely presented quasi-coherent sheaf on
$Y$, and therefore its support $\on{supp}(F(\calO_Z))$ is closed in $Y$. Let $V=Y-\on{supp}(F(\calO_Z))$ be the open complement.
\begin{lem}
The natural open embedding $j:V\to Y$ is quasi-compact.
\end{lem}
\begin{proof}Let $\calI=\calA nn(F(\calO_Z))$ be the annihilator of $F(\calO_Z)$, so that the underlying topological space of $\calO_Y/\calI$ in $Y$ coincides with $\on{supp}(F(\calO_Z))$. Again, by the above observation, it is enough to show that $\calI$ is locally finitely generated. If we could show that locally $F(\calO_Z)$ is generated by one section, then together with the fact that $F(\calO_Z)$ is locally finitely presented, we can conclude that $\calI$ is locally finitely generated by \cite[Theorem 2.6]{M}. Observe that $F(\calO_Z)$ is a quotient of $F(\calO_X)$. Therefore, it is enough to show that $F(\calO_X)$ is locally generated by one section. However, for any closed point $y\in Y$, let $x$ be the unique closed point in $X$ such that $F(\kappa(x))\cong\kappa(y)$. Then
\[\Hom(F(\calO_X),\kappa(y))\cong \Hom(\calO_X,F^{-1}(\kappa(y)))\cong\Hom(\calO_X,\kappa(x))\cong\kappa(x)\cong\kappa(y).\]
Therefore, $F(\calO_X)$ is locally generated by one section and the lemma follows.
\end{proof}
We need another lemma.
\begin{lem}\label{point case}
Let $\xi$ be a point on $Z$ (not necessarily closed), and let $\kappa(\xi)$ be the push-forward of the structure sheaf of $\xi$ to $X$ (so $\on{supp}(\kappa(\xi))$ is the closure of $\xi$ in $Z$). Then $j^*F(\kappa(\xi))=0$.
\end{lem}
\begin{proof}Let $\overline{\xi}$ be the closure of $\xi$ in $Z$ equipped with the reduced closed subscheme structure (so that $\overline{\xi}$ is integral). Then from the surjective morphism $0=j^*(F(\calO_Z))\to j^*(F(\calO_{\overline{\xi}}))$, we conclude that  $j^*(F(\calO_{\overline{\xi}}))=0$. Now, $\kappa(\xi)$ is a filtered colimit of $\calO_{\overline{\xi}}$. Therefore, $j^*F(\kappa(\xi))=0$.
\end{proof}

Now let $i:U\to X, j:V\to Y$ as above. We claim that $F$
induces an equivalence $j^*\circ F\circ
i_*:\on{Qcoh}(U)\to\on{Qcoh}(V)$. Observe that the functor
$i^*:\on{Qcoh}(X)\to\on{Qcoh}(U)$ induces an equivalence
$\on{Qcoh}(X)/\ker(i^*)\cong\on{Qcoh}(U)$ and
\[
\ker(i^*)=\{\calF\in\on{Qcoh}(X),\ \calF|_U=0\}.
\]
Likewise, we have similarly description for $\on{Qcoh}(V)$. Therefore, it is
enough to prove that $F(\ker(i^*))=\ker(j^*)$. Let
\[
\ker(i^*)^c=\{\calF\in\on{Qcoh}(X),\ \calF|_U=0 \mbox{ and }\calF \mbox{ is locally finitely generated}\}
\]
and similar $\ker(j^*)^c$. Since every object in $\ker(i^*)$ is a filtered colimit of objects in $\ker(i^*)^c$,
it is enough to prove that $F(\ker(i^*)^c)=\ker(j^*)^c$. By
applying the functor $F^{-1}$, it is enough to show that
$F(\ker(i^*)^c)$ is a subcategory of $\ker(j^*)^c$. Let
$\calF\in\ker(i^*)^c$. As it is finitely generated and $X$ is
quasi-compact, there is some closed subscheme $W\to X$, $W\cap
U=\emptyset$, such that $\calF$ is indeed a quasi-coherent sheaf
on $W$. We must prove that $j^*F(\calF)=0$, or
\[j^*F(\calF)\cong
j^*q_*(p^*\calF\otimes\calK)\cong
q_*(p^*\calF\otimes\calK|_{W\times_SV})=0.\] It is enough to show
that $\calK|_{W\times_SV}$=0. This can be proven by the following
reduction. Since $\calK|_{W\times_SV}$ is locally finitely
generated, it is enough to prove for any closed point $x$ in
$W\times_SV$, $\calK\otimes\kappa(x)=0$. Let $\xi=p(x)$ be the image
of $x$ under the projection $p:W\times_SV\to W$. It is enough to
prove that $\calK|_{\xi\times_SV}=0$. Let $\kappa(\xi)$ be the
push-forward of the structure sheaf of $\xi$ to $W$. Since $\xi\to
W$ and $q:\xi\times_SV\to V$ are affine morphisms, we obtain that
\[j^*F(\kappa(\xi))\cong q_*(\calK|_{\xi\times_SV}),\]
and therefore it is enough to prove that $j^*F(\kappa(\xi))=0$. Since $W\subset Z$ as closed subset, $\xi\in Z$ and the assertion follows from Lemma \ref{point case}.

We have shown that $j^*\circ F\circ i_*:\on{Qcoh}(U)\to\on{Qcoh}(V)$ is an
equivalence. Now assume that $U=\spec A$ is affine. By the proof
of the special case, we conclude that there is an isomorphism $f:U\to V$ and $\calK|_{U\times_SV}$ is an invertible sheaf on the graph
$\Gamma_f$. It is easy to see this argument localizes well.
Therefore, the proposition follows.
\end{proof}

It is desirable to generalize Theorem \ref{Aut} to the case where
$\calQ coh(X)$ is replaced by the sheaf of the categories of
quasi-coherent twisted sheaves on $X$ (cf. Example \ref{twist}).
Given a $\bbG_m$-gerbe $\calG$ over $X$, determined by a class
$\alpha\in H^2_{\on{fppf}}(X,\bbG_m)\cong H^2_{et}(X,\bbG_m)$, it is
expected one has the following short exact sequence
\begin{equation}\label{non-split}1\to\calP ic_X\to\GL(\calQ
coh_\calG(X))\to\underline{\on{Aut}}_S^\calG(X)\to 1\end{equation}
where $\underline{\on{Aut}}_S^\alpha(X)$ denotes the
$S$-autmorphisms $f$ of $X$ such that $f^*\calG\cong\calG$. We do
not know how to prove this in full generality. But it seems the
arguments used to prove Theorem \ref{Aut} can be generalized to
prove \eqref{non-split} in the case when there exists a locally free
$\calG$-twisted sheaf of finite rank, i.e., the gerbe $\calG$ comes
from an Azumaya algebra.

\begin{rmk}The important feature of \eqref{non-split} is that the
sequence does not split. It has very important applications in the
representation theory of double loop groups (cf. \cite{FZ1}).
\end{rmk}

\section{The Lie 2-algebra of $\GL(\calC)$}\label{Lie 2-alg gl(C)}
Recall that if $G$ is an algebraic group, then its tangent space at
the unit has a natural structure as a Lie algebra. If $G$ is just a
presheaf of groups over $(\mathbf{Sch}/S)$, then under some mild
conditions (cf. \cite{Dem}), the tangent space of $G$ at the unit
also carries on a structure as a Lie algebra. In \S \ref{2-grps to
Lie 2-algs}, we generalize the constructions to sheaves of 2-groups.
We will write down the conditions for 2-groups parallel to those in
\cite{Dem}. Under these conditions, we will show that the tangent
space at the unit of a 2-group has a structure as a Lie 2-algebra.
For the detailed definitions about Lie 2-algebras, see \S \ref{dic}.

In \S \ref{gl(C)}, we will show
that the 2-group $\GL(\calC)$ we considered in this paper
satisfies these conditions. Therefore, its tangent space at the
unite $\gl(\calC)$ is naturally a Lie 2-algebra.

\subsection{The Lie 2-algebra of a 2-group}\label{2-grps to Lie 2-algs}
This subsection is totally formal and is parallel \cite{Dem}, which
discusses the Lie algebras of a (pre)sheaf of groups. Therefore, in
this subsection, we will concentrate on necessary definitions and
the statements of propositions. Only a few proofs are sketched.

Let us still assume that $S$ is a separated, quasi-compact scheme.
Let $\calX$ be a stack over $(\mathbf{Aff}/S)_{fppf}$. Let us
recall the tangent space of $\calX$.

For a quasi-coherent sheaf $M$ on $S$, let $D_S(M):=\calO_S\oplus
M$ denote the sheaf of $\calO_S$-algebras on $S$, with
multiplication given by $(a,m)\cdot(a',m')=(aa',am'+a'm)$. Let
$I_S(M)=\mathbf{Spec}D_S(M)$. Denote $D_S:=D_S(\calO_S)$ and
$I_S:=I_S(\calO_S)$. Then for any quasi-coherent $\calO_S$-module
$M$, we define
\[
T\calX(M):=\underline{\Hom}_S(I_S(M),\calX),
\]
and the tangent space of $\calX$ over $S$ is
\[
T\calX:=T\calX(\calO_S).
\]
$T\calX(-)$ defines a covariant functor from the category of
quasi-coherent sheaves on $S$ to the (2-)category of stacks over
$(\mathbf{Aff}/S)_{fppf}$. In particular, the map $0\to M$ defines
the zero section 1-morphism $z_M:\calX\to T\calX(M)$ and the map
$M\to 0$ defines the projection 1-morphism $p_M:T\calX(M)\to
\calX$.

Since $\calO_S$ acts on any $M$ by multiplication, $\calO_S$ indeed
acts on $T\calX(M)$. Let us recall that this means that there is a
rule which assigns any $S'=\spec R\to S$, and $r\in R$ a functor
$r:T\calX(M)(S')\to T\calX(M)(S')$, and any $r,r'\in R$ an
isomorphism of functors $a_{r,r'}:rr'\cong r\circ r'$ such that
$1\in R$ is assigned to the identity functor and the two
isomorphisms between $rr'r''$ and $r\circ r'\circ r''$ coincide.
Furthermore, such rule should respect any pullback $S''\to S'$.

Let $M,N$ be two quasi-coherent $\calO_S$-modules. From the
projections $M\oplus N\to M$ and $M\oplus N\to N$, one obtains a
canonical map
\begin{equation}\label{E}
T\calX(M\oplus N)\to T\calX(M)\times_{\calX}T\calX(N).
\end{equation}
\begin{dfn}\label{Cond(E)}
We say that $\calX$ satisfies the Condition (E)\footnote{We use this name because when $\calX$ is a presheaf of sets,
similar condition is called Condition (E) in \cite{Dem}. } if for $M,N$
free $\calO_S$-modules of finite rank, the 1-morphism in (\ref{E}) is an
isomorphism.
\end{dfn}
It is well-known that if $\calX$ is an algebraic stack, then it
satisfies the Condition (E).

If $\calX$ satisfies the condition (E), then $T\calX$ has a
structure as a strictly commutative Picard stack over $\calX$.
Namely, let us choose once and for all a quasi-inverse of (\ref{E}).
From the addition $\calO_S\oplus\calO_S\to\calO_S$, we obtain the
monoidal structure
\begin{equation}\label{addition}
+:T\calX\times_\calX T\calX\cong T\calX(\calO_S\oplus\calO_S)\to T\calX.
\end{equation}
The associativity and the commutativity constraints come from the
associativity and the commutativity of the addition. Furthermore,
since $\calO_S$ acts on $T\calX$, $T\calX$ is indeed an
$\calO_S$-linear strictly commutative Picard stack\footnote{We refer to \S \ref{dic} for the detailed discussion of this
notion. } over $\calX$. Observe that different choices of the quasi-inverses of \eqref{E} will give canonically 1-isomorphic $\calO_S$-linear Picard stack structures on $T\calX$.

Now if $u\in \calX(S)$, we define
$T_u\calX(M):=S\times_\calX T\calX(M)$ and
$T_u\calX:=T_u\calX(\calO_S)$. If $\calX$ satisfies the condition
(E), then $T_u\calX$ is a strictly commutative $\calO_S$-linear
Picard stack over $S$. If $F:\calX\to\calY$ is a 1-morphism of
stacks, then for any $u\in\calX(S)$, it induces an
$\calO_S$-linear 1-homomorphisms of Picard stacks
$T_uF:T_u\calX\to T_{F(u)}\calY$ (cf. Definition \ref{hom}).

From now on, we will assume that $\calX=\calG$ is a stack of
2-groups on $(\mathbf{Aff}/S)_{fppf}$. Let $I:S\to\calG$ be the
unit. Let us denote $\Lie(\calG)(M)=T_I\calG(M)$ for $M$ a
quasi-coherent $\calO_S$-module and denote
$\frakg=\Lie(\calG)=\Lie(\calG)(\calO_S)$. If $\calG$ satisfies the
Condition (E), then $\frakg$ is an $\calO_S$-linear Picard stack
over $S$, and therefore, a 2-group over $S$. On the other hand, we
may regard $\frakg$ as the kernel of the 2-group 1-homomorphism
$p:T\calG\to\calG$ (induces from the natural closed embedding $S\to
I_S$), which gives $\frakg$ a second 2-group structure over $S$. We
have the following proposition parallel to Proposition 3.9 of
\cite{Dem} and can by proven similarly.

\begin{prop}\label{same group structure} The identity morphism of $\frakg$ is a canonical 1-isomorphism of these two 2-group structures on $\frakg$.
\end{prop}

Therefore, in what follows, we will not distinguish these two
2-group structures on $\frakg$.

Let $\calP$ be a strictly commutative $\calO_S$-linear Picard
stack. We denote $\underline{\on{Aut}}_{\calO_S-\on{pic}}(\calP)$
the 2-group of $\calO_S$-linear 1-automorphisms of $\calP$, i.e.
$\underline{\on{Aut}}_{\calO_S-\on{pic}}(\calP)$ is the substack
of $\underline{\Hom}_{\calO_S-\on{pic}}(\calP,\calP)$ (see
Definition \ref{hom}) consisting of those $(F,\tau)$ such that $F$
is a 1-isomorphism.

From the splitting exact sequence
\[
1\to\frakg\stackrel{i_p}{\to} T\calG\substack{p\\\rightleftharpoons\\ z}\calG\to 1.
\]
We obtain a 1-homomorphism of 2-groups
\[
\on{Ad}:\calG\to\underline{\on{Aut}}_{\calO_S-\on{pic}}(\frakg)
\]
by sending $\on{Ad}_g(X):=C(z(g),i_p(X))$, where $C$ is the
conjugation action as defined in (\ref{conj}). Observe that there
is a canonical isomorphism $p(\on{Ad}_g(X))\cong I_\calG$.
Therefore, $\on{Ad}_g(X)$ is indeed an element in $\frakg$. It is a routine work to check that $\on{Ad}_g$ has a natural structure as an object in ${\on{Aut}}_{\calO_S-\on{pic}}(\frakg)$, and $\on{Ad}$ has a natural structure as a 1-homomorphism of 2-groups.

To obtain the Lie 2-algebra structure on $\frakg$, we will have to
make an additional assumption on the 2-group $\calG$. Let us begin
with a discussion of the strictly commutative $\calO_S$-linear
Picard stacks. Until the end of this subsection, Picard stacks will
be assumed to be strictly commutative.

Let $\calP$ be an $\calO_S$-linear Picard stack. For any
$f:S'\to S$, let $\calP_{S'}$ denote the restriction of $\calP$ to $S'$. Then there is an $\calO_S$-linear 1-homomorphism of Picard
stacks
\[
\calP\otimes_{\calO_S} f_*\calO_{S'}\to f_*\calP_{S'}.
\]
In particular, if $S'=I_S(M)$, we obtain the
\[
\calP\otimes_{\calO_S} T\calO_S(M)\to T\calP(M).
\]

Following \cite{Dem}, we define
\begin{dfn}\label{good}
An $\calO_S$-linear Picard stack $\calP$ is \emph{good} if the
above 1-homomorphism is a 1-isomorphism for any free
$\calO_S$-module $M$ of finite rank. A stack of 2-groups $\calG$
is \emph{good} if it satisfies the Condition (E), and
$\frakg=\Lie(\calG)$ is a good $\calO_S$-linear Picard stack in
the above sense.
\end{dfn}

\noindent\emph{Caution.} Observe that for an $\calO_S$-linear
Picard stack $\calP$, if it is good as an $\calO_S$-linear Picard
stack, then it is good as a 2-group. However, the vice versa does
not necessarily hold. In what follows, if we say an
$\calO_S$-linear Picard stack $\calP$ good, we mean it is good as
$\calO_S$-linear Picard stacks.

\medskip

From \S \ref{dic}, we associate each $\calO_S$-linear Picard stack
$\calP$ a 2-term complex of $\calO_S$-modules $\calP^\flat$ (over
$(\mathbf{Aff}/S)_{fppf}$). One can easily show that if
$\calP^\flat$ has quasi-coherent cohomology sheaves, then $\calP$
is a good $\calO_S$-linear Picard stack. Therefore, if $\calG$ is
an algebraic stack of 2-groups, then $\calG$ is a good 2-group.
From now on, let us assume that our stack of 2-groups $\calG$ is
good. We have the following proposition parallel Proposition 4.5
and Theorem 4.7 of \cite{Dem}. Although the proof is also parallel, we nevertheless sketch it due to its importance.

\begin{prop}\label{LieGL=gl} If $\calP$ is a good $\calO_S$-linear
Picard stack, then we have the following canonical 1-isomorphism
\[
\rho:\Lie(\underline{\on{Aut}}_{\calO_S-\on{pic}}(\calP))\cong\underline{\Hom}_{\calO_S-\on{pic}}(\calP,\calP).
\]
In addition, the 2-group
$\underline{\on{Aut}}_{\calO_S-\on{pic}}(\calP)$ is a good
2-group.
\end{prop}
\begin{proof}(Sketch.) Let
$\underline{\on{Aut}}_{\calO_S-\on{pic}}(\calP)\to\underline{\Hom}_{\calO_S-\on{pic}}(\calP,\calP)$
be the natural inclusion. It is easy to see that if $\calP$
satisfies the Condition (E), then both stacks also satisfy the
condition. Then the inclusion induces a canonical 1-isomorphism
\[
\Lie(\underline{\on{Aut}}_{\calO_S-\on{pic}}(\calP))\cong T_{\on{Id}}\underline{\Hom}_{\calO_S-\on{pic}}(\calP,\calP).
\]
Next, we construct a 1-homomorphism
\[
\rho':T_{\on{Id}}\underline{\Hom}_{\calO_S-\on{pic}}(\calP,\calP)\to\underline{\Hom}_{\calO_S-\on{pic}}(\calP,\Lie(\calP)),
\]
where the $\calO_S$-linear Picard stack structure on $\Lie(\calP)$
comes from the fact that it is the kernel of  the $\calO_S$-linear
Picard stacks morphism $p:T\calP\to\calP$\footnote{Observe that
since $\calP$ satisfies the Condition (E), $\Lie(\calP)$ has another
$\calO_S$-linear Picard stack structure coming from
\eqref{addition}. A priori, there is no reason that these two
$\calO_S$-linear structures are the same although the identity
morphism on $\Lie(\calP)$ is a canonical 1-isomorphism of the
underlying Picard stacks, by Proposition \ref{same group structure}.
However, as explained in \cite{Dem}, if $\calP$ is good, then these
two $\calO_S$-linear structure are also the same, i.e. the identity
morphism on $\Lie(\calP)$ is a canonical $\calO_S$-linear
1-isomorphism of the $\calO_S$-linear Picard stacks.}.

Let $F\in
T_{\on{Id}}\underline{\Hom}_{\calO_S-\on{pic}}(\calP,\calP)(S')$.
Then $F$ is a 1-homomorphism in
$\Hom_{\calO_{I_{S'}}-\on{pic}}(\calP_{I_{S'}},\calP_{I_{S'}})$
together with an isomorphism $\tilde{p}(F)\cong\on{Id}$ where
\[
\tilde{p}:\Hom_{\calO_{I_{S'}}-\on{pic}}(\calP_{I_{S'}},\calP_{I_{S'}})\to\Hom_{\calO_{S'}-\on{pic}}(\calP_{S'},\calP_{S'})
\]
is the natural restriction along the closed embedding $S'\to
I_S'$. We have the natural $\calO_S$-linear 1-homomorphism
$z:\calP_{S'}\to T\calP_{S'}$. We define $\rho'(F):\calP_{S'}\to
T\calP_{S'}$ by
\begin{equation}\label{rho'}
\rho'(F)(x)=F(z(x))-z(x)
\end{equation}
for $x\in\calP_{S'}$.
It is clear that $\rho'(F)$ has a canonical structure as an
$\calO_{S'}$-linear 1-homomorphism of $\calO_{S'}$-linear Picard
stacks. Let $p:T\calP_{S'}\to\calP_{S'}$ be the projection
1-homomorphism. Since \[p(\rho'(F)(x))\cong p(F(z(x)))-pz(x)\cong
\tilde{p}(F)(p(z(x)))-x\cong 0\] canonically, we indeed obtain an
$\calO_{S'}$-linear 1-homomorphism
$\calP_{S'}\to\Lie(\calP_{S'})$. We leave it to readers to verify
that $\rho'$ has a canonical structure as $\calO_S$-linear
1-homomorphism of Picard stacks.

Finally, if $\calP$ is good, then the canonical 1-homomorphism
$\calP\to\Lie(\calP)$ is a 1-isomorphism. And the first assertion
of proposition follows. The second assertion follows from the fact
that if $\calP$ is good, then
$\underline{\Hom}_{\calO_S-\on{pic}}(\calP,\calP)$ is good.
\end{proof}

This proposition allows us to construct an $\calO_S$-bilinear
1-homomorphism, called the Lie bracket,
\[
[-,-]:\frakg\times_S\frakg\to\frakg
\]
for a good 2-group $\calG$. Namely, from the 1-homomorphism
\[
\on{Ad}:\calG\to\underline{\on{Aut}}_{\calO_S-\on{pic}}(\frakg),
\]
we obtain
\[
\on{ad}:=T_I\on{Ad}:\frakg\to\Lie(\underline{\on{Aut}}_{\calO_S-\on{pic}}(\frakg))\cong\underline{\Hom}_{\calO_S-\on{pic}}(\frakg,\frakg).
\]
Then define
\[
[X,Y]:=\on{ad}(X)(Y), \mbox{ for } X,Y\in\frakg.
\]
\begin{prop}\label{skew-symmetric} There is a canonical 2-isomorphism $s$ of
$\calO_S$-bilinear 1-homomorphisms $\frakg\times\frakg\to\frakg,
(X,Y)\to[X,Y]$ and $(X,Y)\to-[Y,X]$, such that the following
diagram commutes
\[
\xymatrix{[X,Y]\ar^{s_{X,Y}}[r]\ar@{=}[d]&-[Y,X]\ar^{-s_{Y,X}}[d]\\
[X,Y]&-(-[X,Y])\ar^{a_{-1,-1}}[l]. }
\]
where $a_{r,r'}:r\circ
r'\cong rr', r,r'\in\calO_S$ is as in Definition \ref{R-linear
Pic} (a).
\end{prop}
\begin{proof}(Sketch.)Observe that we have the following morphisms of
$\calO_S$-algebras.
\[
\calO_S[\varepsilon]/(\varepsilon^2)\stackrel{q(\varepsilon)=\varepsilon_1\varepsilon_2}{\longrightarrow}\calO_S[\varepsilon_1,\varepsilon_2]/(\varepsilon_1^2,\varepsilon_2^2)\stackrel{j(\varepsilon_1\varepsilon_2)=0}{\longrightarrow}\calO_S[\varepsilon_1,\varepsilon_2]/(\varepsilon_1^2,\varepsilon_2^2,\varepsilon_1\varepsilon_2).
\]
Therefore, we obtain the 1-homomorphisms of 2-groups
\begin{equation}\label{q}
\frakg\stackrel{i_p}{\to}T\calG\stackrel{q}{\to}\underline{\Hom}(I_S\times_SI_S,\calG)\stackrel{j}{\to}\underline{\Hom}(I_S(\calO_S\oplus\calO_S),\calG).
\end{equation}

Let
\begin{equation}\label{pr}
\pr_i^*:T\calG=\underline{\Hom}(I_S,\calG)\to\underline{\Hom}(I_S\times
I_S,\calG)
\end{equation}
be the 1-morphism induced from the projection of $I_S^2$ to its
$i$th factor. We leave it to readers to verify the following lemma,
using the definition of Lie bracket and the construction given in
the previous proof (in particular the formula (\ref{rho'})).

\begin{lem}If $\calG$ is good, then for any $X,Y\in\frakg$, there is a canonical isomorphism
\begin{equation}\label{Lie bracket}
\begin{split}
(qi_p)([X,Y])&\cong
\pr_1^*(X)\pr_2^*(Y)\sigma(\pr_1^*(X))\sigma(\pr_2^*(Y))\\
&\cong\pr_2^*(X)\pr_1^*(Y)\sigma(\pr_2^*(X))\sigma(\pr_1^*(Y)).
\end{split}
\end{equation}
\end{lem}
In order to prove of the proposition, we need one more remark. If
$\calP$ is an $\calO_S$-linear Picard groupoid, then there is a
canonical isomorphism $-x\cong\sigma(x), x\in\calP$ such that the
following diagram commutes.
\[
\xymatrix{-(-x)\ar^{\cong}[rr]\ar_{a_{-1,-1}}[dr]&&\sigma^2(x)\\
&x\ar_{e}[ur]&},
\]
where $e:\on{Id}\cong\sigma^2$ is the canonical
isomorphism induced by (\ref{x=sigma2(x)}).

Now we prove the proposition. We define $s$ to be the composition
of the following canonical isomorphisms (using the above remark)
\[\begin{split}(qi_p)(s_{X,Y}):(qi_p)([X,Y])\cong& \pr_1^*(X)\pr_2^*(Y)\sigma(\pr_1^*(X))\sigma(\pr_2^*(Y))\\
\cong&\sigma(\pr_1^*(Y)\pr_2^*(X)\sigma(\pr_1^*(Y))\sigma(\pr_2^*(X))) \\
\cong&\sigma ((qi_p)([Y,X]))\cong (qi_p)(-[Y,X])\end{split}\]
Using the above remark again, and Corollary \ref{for (i)}, one can
show that $s_{X,Y}$ indeed satisfies
\[a_{-1,-1}(-s_{Y,X})s_{X,Y}=\on{id}_{[X,Y]}.\] We leave it to
readers to verify that $s$ is a 2-isomorphism of
$\calO_S$-bilinear 1-homomorphisms of $\calO_S$-linear Picard
stacks.
\end{proof}

According to the proof the above proposition, we can easily obtain

\begin{prop}\label{func1}Let $F:\calG\to\calH$ be a 1-homomorphism of good
2-groups. Then $dF=T_eF:\Lie(\calG)\to\Lie(\calH)$ satisfies:
there is a 2-canonical isomorphism of $\calO_S$-bilinear
1-homomorphims
\[
\theta:dF([-,-])\cong [dF(-),dF(-)],
\]
such that the following diagram commutes
\[
\xymatrix{dF([X,Y])\ar^{\theta}[rr]\ar_{dF(s)}[d]&&[dF(X),dF(Y)]\ar^{s}[d]\\
dF(-[Y,X])\ar^{\tau}[r]&-dF([Y,X])\ar^(0.4){-\theta}[r]&-[dF(Y),dF(X)],}
\]
where $\tau$ is the constraint as in Definition \ref{hom}.
\end{prop}

The following proposition is parallel Proposition 4.8 of
\cite{Dem}. Although the proof is also parallel, we nevertheless sketch it. 

\begin{prop}\label{jacobiator}Let $\calP$ be a good $\calO_S$-linear
Picard stack, so that we have the following canonical
1-isomorphism
\[
\rho:\Lie(\underline{\on{Aut}}_{\calO_S-\on{pic}}(\calP))\cong\underline{\Hom}_{\calO_S-\on{pic}}(\calP,\calP).
\]
Then there is a 2-canonical isomorphism of two $\calO_S$-bilinear
1-homomorphisms from
$\Lie(\underline{\on{Aut}}_{\calO_S-\on{pic}}(\calP))\times\Lie(\underline{\on{Aut}}_{\calO_S-\on{pic}}(\calP))$
to $\underline{\Hom}_{\calO_S-\on{pic}}(\calP,\calP)$
\[
\theta:\rho([X,Y])\cong\rho(X)\rho(Y)-\rho(Y)\rho(X), \ \ X,Y\in\Lie(\underline{\on{Aut}}_{\calO_S-\on{pic}}(\calP)),
\]
such that the following diagram commutes.
\[
\xymatrix{\rho([X,Y])\ar^{\theta}[d]\ar^{\rho(s_{X,Y})}[r]&\rho(-[Y,X])\ar^{\tau}[r]& -\rho([Y,X])\ar^{-\theta}[d]\\
\rho(X)\rho(Y)-\rho(Y)\rho(X)\ar^{\cong}[rr]&&-(\rho(Y)\rho(X)-\rho(X)\rho(Y)).}
\]
\end{prop}
\begin{proof} Let
$X\in\Lie(\underline{\on{Aut}}_{\calO_S-\on{pic}}(\calP))(S')$.
Then it gives an object
\[1+\varepsilon\rho(X)\in\Hom_{\calO_{S'}-\on{pic}}(\calP_{S'},\calP_{S'})\oplus\varepsilon\Hom_{\calO_{S'}-\on{pic}}(\calP_{S'},\calP_{S'}). \]
It is easy to see under the natural equivalence of Picard
groupoids
\[\Hom_{\calO_{S'}-\on{pic}}(\calP_{S'},\calP_{S'})\oplus\varepsilon\Hom_{\calO_{S'}-\on{pic}}(\calP_{S'},\calP_{S'})\cong\Hom_{\calO_{I_{S'}}-\on{pic}}(\calP_{I_{S'}},\calP_{I_{S'}})\]
$1+\varepsilon\rho(X)$ indeed belongs to
$\underline{\on{Aut}}_{\calO_S-\on{pic}}(\calP)(I_{S'})\subset\Hom_{\calO_{I_{S'}}-\on{pic}}(\calP_{I_{S'}},\calP_{I_{S'}})$,
and is the object corresponding to $X$.

Now if
$X,Y\in\Lie(\underline{\on{Aut}}_{\calO_S-\on{pic}}(\calP))(S')$,
then in
$\underline{\on{Aut}}_{\calO_S-\on{pic}}(\calP)(I_{S'}\times_{S'}I_{S'})$,
we have
\[\begin{array}{ll}&(1+\varepsilon_1\rho(X))(1+\varepsilon_2\rho(Y))\sigma(1+\varepsilon_1\rho(X))\sigma(1+\varepsilon_2\rho(Y))\\
\cong&(1+\varepsilon_1\rho(X))(1+\varepsilon_2\rho(Y))(1-\varepsilon_1\rho(X))(1-\varepsilon_2\rho(Y))\\
\cong&1+\varepsilon_1\varepsilon_2(\rho(X)\rho(Y)-\rho(Y)\rho(X)),\end{array}\]
since $\varepsilon_1^2=\varepsilon_2^2=0$. On the other hand, it
is easy to see that for any $X\in\frakg$,
$(qi_p)(X)\in\underline{\on{Aut}}_{\calO_S-\on{pic}}(\calP)(I_{S'}\times_{S'}I_{S'})$
is just $1+\varepsilon_1\varepsilon_2\rho(X)$, where
\[qi_p:\Lie(\underline{\on{Aut}}_{\calO_S-\on{pic}}(\calP))(S')\to\underline{\on{Aut}}_{\calO_S-\on{pic}}(\calP)(I_{S'}\times_{S'}I_{S'})\]
is as defined in the proof of Proposition \ref{skew-symmetric}.
Therefore,, there is a canonical isomorphism
\[1+\varepsilon_1\varepsilon_2\rho([X,Y])=(qi_p)([X,Y])\cong 1+\varepsilon_1\varepsilon_2(\rho(X)\rho(Y)-\rho(Y)\rho(X)),\]
or
\[\rho([X,Y])\cong\rho(X)\rho(Y)-\rho(Y)\rho(X)\]
canonically. It is easy to see from the construction that this
canonical isomorphism is an isomorphism of $\calO_S$-bilinear
1-homomorphisms between $\calO_S$-linear Picard stacks, and the
commutative diagram holds.
\end{proof}

Now we apply the above two propositions to the case
$\calP=\frakg=\Lie(\calG)$ for a good 2-group $\calG$. We thus
obtain that for any $X,Y,Z\in\frakg$, a canonical isomorphism of
$\calO_S$-trilinear 1-homomorphisms from
$\frakg\times\frakg\times\frakg\to\frakg$
\[
\on{ad}([X,Y])(Z)\cong\on{ad}(X)(\on{ad}(Y)(Z))-\on{ad}(Y)(\on{ad}(X)(Z)).
\]
Or in other words,
\[
j'_{X,Y,Z}:[[X,Y],Z]\cong[X,[Y,Z]]-[Y,[X,Z]].
\]
This will give us the following canonical isomorphism, denoted by
$j_{X,Y,Z}$:
\[
\begin{split}
&[[X,Y],Z]+[[Y,Z],X]+[[Z,X],Y]\\
\stackrel{j'}{\longrightarrow}\ &[X,[Y,Z]]-[Y,[X,Z]]+[[Y,Z],X]+[[Z,X],Y]\\
\stackrel{s}{\longrightarrow}\ &[X,[Y,Z]]-[Y,[X,Z]]-[X,[Y,Z]]+[Y,[X,Z]]\longrightarrow
0.
\end{split}
\]

\begin{prop}\label{alt}$j$ satisfies the following commutative diagram
\[
\xymatrix{[[X,Y],Z]+[[Y,Z],X]+[[Z,X],Y]\ar^(.75){j_{X,Y,Z}}[rr]\ar_{[s_{X,Y},Z]+[s_{Y,Z,X]}+[s_{Z,X},Y]}[d]&&  0\ar@{=}[dd]\\
[-[Y,X],Z]+[-[Z,Y],X]+[-[X,Z],Y]\ar^\cong[d]&& \\
 -[[Y,X],Z]-[[X,Z],Y]-[[Z,Y],X]\ar^(.75){-j_{Y,Z,X}}[rr]&&0.}
\]
\end{prop}
\begin{proof}It is equivalent to show the commutativity of the
following diagram
\[
\xymatrix{[[X,Y],Z]\ar^{j'_{X,Y,Z}}[d]\ar^{[s_{X,Y},Z]}[r]&[-[Y,X],Z]\ar^{\cong}[r]&-[[Y,X],Z]\ar^{-j'_{Y,X,Z}}[d]\\
[X,[Y,Z]]-[Y,[X,Z]]\ar^{\cong}[rr]&&-([Y,[X,Z]]-[X,[Y,Z]]). }
\]
which follows directly from Proposition \ref{jacobiator}.
\end{proof}

In order to state the following key lemma, let us first recall the
definition $(x,y,z)$ given by (\ref{tricomm}) and the statement of
Corollary \ref{tri-identity}. Next, let
\[
m:\calO_S[\varepsilon]/(\varepsilon^2)\to\calO_S[\varepsilon_1,\varepsilon_2,\varepsilon_3]/(\varepsilon_1^2,\varepsilon_2^2,\varepsilon_3^2)
\]
be the $\calO_S$-algebra homomorphism defined by
$m(\varepsilon)=\varepsilon_1\varepsilon_2\varepsilon_3$. We thus
obtain a 1-homomorphism of 2-groups
\begin{equation}\label{m}
\frakg\stackrel{i_p}{\to} T\calG\stackrel{m}{\to}\underline{\Hom}(I_S\times I_S\times I_S,\calG).
\end{equation}
Let $\pr_i^*$ be the 1-morphism as in \eqref{pr} induced from the
projection of $I_S^3$ to its $i$th factor.
\begin{lem}\label{jacobiator2}For $X,Y,Z\in\frakg$, there is a canonical isomorphism in
$\underline{\Hom}(I_S\times I_S\times I_S,\calG)$,
\[
\xi_{X,Y,Z}:(mi_p)([[X,Y],Z])\cong((\pr_1^*X,\pr_2^*Y),(\pr_3^*Z)^{\pr_2^*Y}),
\]
such that the following diagram commutes
\[
\xymatrix{(mi_p)([[X,Y],Z]+[[Y,Z],X]+[[Z,X],Y])\ar[d]\ar^(0.75){(mi_p)(j_{X,Y,Z})}[r]&(mi_p)(0)\ar^{\cong}[dd]\\
(mi_p)([[X,Y],Z])+(mi_p)[[Y,Z],X])+(mi_p)([[Z,X],Y])\ar_{\xi_{X,Y,Z}\xi_{Y,Z,X}\xi_{X,Y,Z}}[d]&\\
(\pr_1^*X,\pr_2^*Y,\pr_3^*Z)\ar^\cong[r]&I_{\underline{\Hom}(I_S\times{I_S}\times{I_S},\calG)}.}
\]
\end{lem}
\begin{proof}(Sketch.) We only sketch the proof of the first part of the lemma. The proof for the second part is elementary but tedious, and therefore is omitted. First, using the canonical isomorphism (\ref{Lie bracket}), we have the following canonical isomorphism
\[
(mi_p)([[X,Y],Z])\cong((\pr_1^*X,\pr_2^*Y),\pr_3^*Z).
\]
Next, we shall that there is a canonical isomorphism in $\underline{\Hom}(I_S\times I_S\times I_S,\calG)$
\[
((\pr_1^*X,\pr_2^*Y),\pr_3^*Z)\cong ((\pr_1^*X,\pr_2^*Y),(\pr_3^*Z)^{\pr_2^*Y}).
\]
To this end, we need
\begin{sublem}Recall the map (\ref{q}). Then for any $X,Y\in\frakg$, there is a canonical isomorphism in $\underline{\Hom}(I_S\times I_S,\calG)$
likewise, recall the map (\ref{m}). Then for any $X,Y\in\frakg$, there is a canonical isomorphism in $\underline{\Hom}(I_S\times I_S\times I_S,\calG)$
\[(mi_p)(X)\pr_2^*(Y)\cong \pr_2^*(Y)(mi_p)(X).\]
\end{sublem}
By this lemma, we thus have the canonical isomorphisms
\[\begin{split}&((\pr_1^*X,\pr_2^*Y),(\pr_3^*Z)^{\pr_2^*Y})\\
\cong&(\pr_1^*X,\pr_2^*Y)\pr_2^*Y\pr^*_3Z\sigma(\pr_2^*Y)\sigma(\pr_1^*X,\pr_2^*Y)\pr_2^*Y\sigma(\pr_3^*Z)\sigma(\pr_2^*Y)\\
                                           \cong&\pr_2^*Y(\pr_1^*X,\pr_2^*Y)\pr_3^*Z\sigma((\pr_1^*X,\pr_2^*Y))\sigma(\pr_3^*Z)\sigma(\pr_2^*Y)\\
                                           \cong&\pr_2^*((\pr_1^*X,\pr_2^*Y),\pr_3^*Z)\sigma(\pr_2^*Y)\\
                                          \cong&((\pr_1^*X,\pr_2^*Y),\pr_3^*Z).
\end{split}\]
The first part of the lemma follows.
\end{proof}
There are some consequences of this lemma.

\begin{prop}\label{cyc}$j$ satisfies the following commutative diagram.
\[
\xymatrix{[[X,Y],Z]+[[Y,Z],X]+[[Z,X],Y]\ar^(.75){j_{X,Y,Z}}[rr]\ar^{\cong}[d]&& 0\ar@{=}[d]\\
[[Y,Z],X]+[[Z,X],Y]+[[X,Y],Z]\ar^(.75){j_{Y,Z,X}}[rr]&& 0. }
\]
for any $X,Y,Z\in\frakg$.
\end{prop}
\begin{proof}By Lemma \ref{for (ii)}, it is enough to prove the
induced isomorphism $[[X,Y],Z]\to[[X,Y],Z]$ is the identity map.
But this follows from Corollary \ref{tri-identity} and Lemma
\ref{jacobiator2}.
\end{proof}

\begin{prop}\label{func2}Let $F:\calG\to\calH$ be a 1-homomorphism of good
2-groups. Then the canonical 2-isomorphism
\[
\theta:dF([-,-])\cong [dF(-),dF(-)]
\]
as in Proposition \ref{func1} satisfies the following condition
\[
\begin{CD}dF(\sum[[X,Y],Z])@>dF(j)>>0\\
@VV\theta V @ | \\
\sum[[dF(X),dF(Y)],dF(Z)]@>j>>0,
\end{CD}
\]
where $\sum$ denotes the sum over the cyclic permutations of
$X,Y,Z$.
\end{prop}
\begin{proof}We have the following observation. If $\Phi:\calG_1\to\calG_2$ is a 1-homomorphism of 2-groups, then for any $x,y,z\in\calG_1$, the following diagram commutes.
\[
\xymatrix{\Phi(x,y,z)\ar[r]\ar[d]&\Phi(I_{\calG_1})\ar[d]\\
(\Phi(x),\Phi(y),\Phi(z))\ar[r]&I_{\calG_2}.}
\]
Now the proposition follows from the above observation and Lemma \ref{jacobiator2}.
\end{proof}

\begin{prop}\label{jac identity}Let $\calP$ be a good $\calO_S$-linear
Picard stack. Under the canonical 1-isomorphism
\[
\rho:\Lie(\underline{\on{Aut}}_{\calO_S-\on{pic}}(\calP))\cong\underline{\Hom}_{\calO_S-\on{pic}}(\calP,\calP),
\]
the isomorphism
$\rho(j_{X,Y,Z}):\rho([[X,Y],Z]+[[Y,Z],X]+[[Z,X],Y])\cong
\rho(0)\cong 0$ is the same as the isomorphism
\[
\begin{split}&\rho([[X,Y],Z]+[[Y,Z],X]+[[Z,X],Y])\\
\cong\ &\sum\rho([X,Y])\rho(Z)-\rho(Z)\rho([X,Y]) \quad\quad \mbox{By Proposition \ref{jacobiator}}\\
\cong\ &\sum(\rho(X)\rho(Y)\rho(Z)-\rho(Y)\rho(X)\rho(Z)-(\rho(Z)\rho(X)\rho(Y)-\rho(Z)\rho(Y)\rho(X))\\
\cong\ & 0,
\end{split}
\]
where $\sum$ denotes the sum over the cyclic permutations of
$X,Y,Z$.
\end{prop}
The proof uses the same argument as in the proof of Proposition \ref{jacobiator}, once Lemma \ref{jacobiator2} is known.

Let us summarize the above construction into the following
theorem.
\begin{thm} Let $[-,-]:\frakg\times\frakg\to\frakg$ be the Lie
bracket we construct as above, and $s,j$ are the natural
isomorphisms. Then these data give $\frakg$ a structure of pseudo
Lie 2-algebra (cf. Definition \ref{Lie 2-alg}). The assignment
$\calG\to\Lie(\calG)$ is a functor from the 2-category of good
2-groups to the 2-category of pseudo Lie 2-algebras.
\end{thm}
\begin{proof}The compatibility condition (i) in the Definition
\ref{Lie 2-alg} is given in Proposition \ref{skew-symmetric}; (ii)
is given in Proposition \ref{cyc}; (iii) is given in Proposition
\ref{alt}; and (iv) is the combination of Proposition \ref{func2}
and Proposition \ref{jac identity}.

If $F:\calG\to\calH$ is a 1-homomorphism of good 2-groups, then
$dF$ is a 1-homomorphism of pseudo Lie 2-algebras by Proposition
\ref{func1} and Proposition \ref{func2}.
\end{proof}

\subsection{The Lie 2-algebra $\gl(\calC)$}\label{gl(C)}
We will apply the general construction presented in the previous
subsection to the case of $\GL(\calC)$. We begin with the proof of
\begin{prop}\label{cond(E)}Let $\calC$ be a quasi-coherent sheaf of abelian categories over $S$. Then
$\GL(\calC)$ satisfies the Condition $(E)$.
\end{prop}
\begin{proof}
Let $S'=\spec R$ be over $S$, and let $M,N$ are two free
$R$-modules of finite rank. We have the closed embeddings
\[
\xymatrix{&I_{S'}(M)\ar^{i_M}[dr]&\\
S'\ar^{j_M}[ur]\ar_{j_N}[dr]&&I_{S'}(M\oplus N).\\
&I_{S'}(N)\ar_{i_N}[ur]&}
\]
We thus obtain a 2-group homomorphism
\[
\GL(\calC)(D_{S'}(M\oplus N))\stackrel{i_M^*\times i_N^*}{\longrightarrow} \GL(\calC)(D_{S'}(M))\times_{\GL(\calC)(S')}\GL(\calC)(D_{S'}(N)).
\]
To prove the proposition, it is enough to show that it is an
isomorphism of 2-groups (i.e. equivalence of categories).  We
therefore could replace $S$ by $S'=\spec R$. It is clear that we
could assume that $\calC_{D_S(M)}$ (resp. $\calC_{D_S(N)}$, resp.
$\calC_{D_S(M\oplus N)}$) is the base change of the category
$\calC_R$ from $R$ to $D_S(M)$ (resp. to $D_S(N)$, resp.
$D_S(M\oplus N)$).

We will discuss the following more general settings. Let $\calC$ be
an abelian category over $S=\spec R$. Let $i:A'\to A, j:A''\to A$ be
$R$-algebra homomorphisms. Denote $A''':=A'\times_AA''$ and the
natural projections $p:A'''\to A', q:A'''\to A''$. We ask when
natural map
\begin{equation}\label{tildeF}
\tilde{F}:=p^*\times q^*:\GL_{A'''}(\calC_{A'''})\to
\GL_{A'}(\calC_{A'})\times_{\GL_A(\calC_A)}\GL_{A''}(\calC_{A''})
\end{equation}
is an isomorphism of 2-groups.

To achieve this, we should first analyze the corresponding abelian
categories. We have the functor
\[
F=p^*\times q^*:\calC_{A'''}\to
\calC_{A'}\times_{\calC_A}\calC_{A''}.
\]
Observe that $\calC_{A'}\times_{\calC_A}\calC_{A''}$ is indeed an
abelian category over $A'''$ (in an obvious way), and that $F$ is an
$A'''$-linear right exact functor. We claim that $F$ has a right
adjoint, which is exact.

Let us recall that objects in $\calC_{A}$ are of the form
$(X,\alpha)$, where $X\in\calC_R$ and $\alpha:A\to\End_{\calC_R}X$
that recovers the original $R$-structure on $X$ when composed with
$R\to A$. Likewise, objects in $\calC_{A'},\calC_{A''},\calC_{A'''}$
have similar descriptions. Now let $((X,\alpha),(Y,\beta),\varphi)$
be an object in $\calC_{A'}\times_{\calC_A}\calC_{A''}$. This means
that $(X,\alpha)$ is an object in $\calC_{A'}$, $(Y,\beta)$ is an
object in $\calC_{A''}$ and
\[\varphi:i^*(X,\alpha)\cong j^*(Y,\beta).\]
Let us choose for each $((X,\alpha),(Y,\beta),\varphi)$ and object
$(Z,\gamma)\in\calC_{A}$ and an isomorphism $(Z,\gamma)\cong
i^*(X,\alpha)\cong j^*(Y,\beta)$ in $\calC_A$. Recall that we have
the natural adjunction maps $(X,\alpha)\to i_*i^*(X,\alpha)$ and
$(Y,\beta)\to j_*j^*(Y,\beta)$. Therefore, in $\calC_R$, we have
the natural maps $X\to Z$ and $Y\to Z$ and we can form the fiber
product $X\times_ZY$, on which $A'''=A'\times_AA''$ acts via
$\alpha\times\beta$. It is easy to check this construction is
valid and therefore $(X\times_ZY,\alpha\times\beta)$ is an object
in $\calC_{A'''}$. In addition, if $(f,g)$ is a morphism from
$((X,\alpha),(Y,\beta),\varphi)$ to
$((X',\alpha'),(Y',\beta'),\varphi')$, then $f\times g$ is a
morphism from $(X\times_{Z}Y,\alpha\times\beta)$ to
$(X'\times_{Z'}Y',\alpha'\times\beta')$. Therefore, we obtain a
functor $G:\calC_{A'}\times_{\calC_A}\calC_{A''}\to\calC_{A'''}$.
From the definition, it is easy to see that $G$ is the right
adjoint of $F$.

\medskip

Now we return to analyze when $\tilde{F}$ is an isomorphism.
Observe that there is a natural equivalence
\[
\GL_{A'''}(\calC_{A'}\times_{\calC_A}\calC_{A''})\cong \GL_{A'}(\calC_{A'})\times_{\GL_A(\calC_A)}\GL_{A''}(\calC_{A''}).
\]
Let $\tilde{F}$ also denote the composition of the functor
$\tilde{F}$ in \eqref{tildeF} with a quasi-inverse of the above
equivalence. Then by Lemma \ref{f^*F=(f^*F)f^*}, we have for any
$g\in\GL_{A'''}(\calC_{A'''})$,
\[
F\circ g\cong \tilde{F}(g)\circ F:\calC_{A'''}\to\calC_{A'}\times_{\calC_A}\calC_{A''}.
\]
By the adjunction, we also obtain
\[
g\circ G\cong
G\circ\sigma(\tilde{F}(\sigma(g))):\calC_{A'}\times_{\calC_A}\calC_{A''}\to\calC_{A'''}.
\]
Recall that $\sigma$ denotes the inversion 1-anti-homomorphism of a
2-group, so that $\sigma(g)$ means the chosen quasi-inverse of $g$,
etc.

We claim

\begin{lem}\label{criterion}
Assume that the natural adjunction $F\circ G\to\on{Id}$ is an
isomorphism. Let
$\calD=G(\calC_{A'}\times_{\calC_A}\calC_{A''})\subset\calC_{A'''}$.
and assume that there is a functor $H:\calC_{A'''}\to\calD$ and a
natural transform from the functor
$\calC_{A'''}\stackrel{H}{\to}\calD\subset\calC_{A'''}$ to the
identity functor $\on{Id}:\calC_{A'''}\to\calC_{A'''}$ such that for
any $(X,\alpha)\in\calC_{A'''}$, the functorial map
$H(X,\alpha)\to(X,\alpha)$ is an epimorphism. Then $\tilde{F}$ is an
isomorphism.
\end{lem}
\begin{proof}To show that $\tilde{F}$ is an isomorphism, it is enough to show
that $\ker\tilde{F}\cong\on{Id}$ and that $\tilde{F}$ is
essentially surjective.

We first prove $\ker\tilde{F}\cong\on{Id}$. Let us write
$(X,\alpha)\in\calC_{A'''}$ just by $X$ for brevity. Then for any
$m:X\to Y$ a morphism in $\calC_{A'''}$, we have the following
diagram in $\calC_{A'''}$ with horizontal rows being exact
\begin{equation}\label{resolution}\begin{CD}
H(X')@>>> H(X)@>>> X@>>> 0\\
@VVV@VVH(m)V@VVmV \\
H(Y')@>>>H(Y)@>>>Y@>>>0,
\end{CD}\end{equation}
where $X'=\ker (H(X)\to
X)$ and $Y'=\ker(H(Y)\to Y)$. Observe that since $F\circ G\cong
\on{Id}$, $G$ is full and faithful and therefore $\calD$ is a full
subcategory of $\calC_{A'''}$. Therefore, the left square of the
above commutative diagram is indeed a commutative diagram in
$\calD$.

Let $g$ is an $A'''$-linear auto-equivalence of $\calC_{A'''}$ such
that $\tilde{F}(g)\cong\on{Id}$. Since $gG\cong
G\sigma(\tilde{F}(\sigma(g)))$, we have $g(\calD)\subset\calD$, and
therefore
\[g|_{\calD}\cong GFg|_{\calD}\cong
G\tilde{F}(g)F|_{\calD}\cong GF|_{\calD}\cong\on{Id}|_{\calD}\]
Therefore, the action of $g$ on the left square of the diagram
\eqref{resolution} is isomorphic to the action of the identity
functor. Therefore, $g$ is isomorphic to the identity functor.
Furthermore, it is clear any automorphism of $g$ which maps the
identity automorphism of $\on{Id}$ under $\tilde{F}$ must be the
identity automorphism itself. This proves that
$\ker\tilde{F}\cong\on{Id}$.

To prove $\tilde{F}$ is essentially surjective, let $g$ be an
$A'''$-linear auto-equivalence of
$\calC_{A'}\times_{\calC_A}\calC_{A''}$. Then define an
$A'''$-linear auto-equivalence $\tilde{g}$ of $\calC_{A'''}$ as
follows. First define $\tilde{g}|_{\calD}=GgF|_{\calD}$. Since
$FG\cong\on{Id}$, it is easy to see that $\tilde{g}$ is an
$A'''$-linear auto-equivalence of $\calD$. Then for general
$m:X\to Y$ a morphism in $\calC_{A'''}$, $\tilde{g}(m:X\to Y)$ is
defined as the cokernel of
\[\tilde{g}\left(\begin{CD}H(X')@>>> H(X)\\
@VVV@VVV \\
H(Y')@>>>H(Y)\end{CD}\right),\] where $X', Y'$ are as in
\eqref{resolution}. It is easy to see that $\tilde{g}$ is
well-defined object in $\GL_{A'''}(\calC_{A'''})$, and its image
under $\tilde{F}$ is $g$.
\end{proof}

To apply this lemma, we need the following two lemmas.

\begin{lem} (i) Assume that $j: A''\to A$ is surjective with
kernel $J$. Then for any
$((X,\alpha),(Y,\beta),\varphi)\in\calC_{A'}\times_{\calC_A}\calC_{A''}$,
\[p^*G((X,\alpha),(Y,\beta),\varphi)\cong (X,\alpha).\]
If in addition $J^2=0$, then there is a surjective map
\[q^*G((X,\alpha),(Y,\beta),\varphi)\to (Y,\beta).\]

\medskip

\noindent(ii) Assume that $i:A'\to A$ and $j:A''\to A$ are
surjective. Then $F\circ G\to\on{Id}$ is an isomorphism, and
therefore $G$ is full and faithful.

\medskip

\noindent(iii) Assume that $i:A'\to A$ and $j:A''\to A$ are
surjective. Let $I=\ker i, J=\ker j$ and assume that $J^2=0$. The
functor $G$ realizes $\calC_{A'}\times_{\calC_A}\calC_{A''}$ as a
full subcategory of $\calC_{A'''}$ consisting of
$(X,\alpha),X\in\calC_R,\alpha:A'''\to\End_{\calC_R} X$ such that
$\alpha(I)X\cap\alpha(J)X=0$, where $\alpha(I)X\cap\alpha(J)X$ is
defined to be the fiber product $\alpha(I)X\times_X\alpha(J)X$ in
$\calC_R$.

\medskip

\noindent(iv) Assumptions are as in (iii) Let $f:R\to A'''$ be the
structural map. Then for any $(X,\alpha)\in\calC_{A'''}$,
$f^*f_*(X,\alpha)\in G(\calC_{A'}\times_{\calC_A}\calC_{A''})$.
\end{lem}
\begin{proof}We first prove (i).
Let
$((X,\alpha),(Y,\beta),\varphi)\in\calC_{A'}\times_{\calC_A}\calC_{A''}$,
and $(Z,\gamma)\cong i^*(X,\alpha)\cong j^*(Y,\beta)\in\calC_A$.
We need to prove that there is an isomorphism
\[
p^*((X\times_{Z}Y,\alpha\times\beta))\cong (X,\alpha).
\]
We have the following pullback diagram in $\calC_R$
\[\begin{CD}
X\times_{Z}Y@>>>X\\
@VVV@VVV\\
Y@>>>Z.
\end{CD}\]
It is easy to see that this gives us the following pullback
diagram in $\calC_{A'''}$
\[\begin{CD}
(X\times_{Z}Y,\alpha\times\beta)@>>>p_*(X,\alpha)\\
@VVV@VVV\\
q_*(Y,\beta)@>>>(pi)_*(Z,\gamma).
\end{CD}\]
Since $p^*$ is right exact, we obtain a pullback diagram in
$\calC_{A'}$ by pulling back the above diagram along $p^*$.
\[\begin{CD}
p^*(X\times_{Z}Y,\alpha\times\beta)@>>>p^*p_*(X,\alpha)\\
@VVV@VVV\\
p^*q_*(Y,\beta)@>>>p^*(pi)_*(Z,\gamma).
\end{CD}\]
Since $j:A''\to A$ is surjective, $p:A'\times_AA''\to A'$ is
surjective. By Lemma \ref{closed embedding}, $p^*p_*(X,\alpha)\cong
(X,\alpha),p^*(pi)_*(Z,\gamma)\cong i_*(Z,\gamma)$. Furthermore, it
is easy to see that
\[p^*q_*(Y,\beta)\cong i_*j^*(Y,\beta)\cong i_*(Z,\gamma).\]
Therefore $p^*((X\times_{Z}Y,\alpha\times\beta))\cong (X,\alpha)$.

Next assume that $J^2=0$. From the map
$(X\times_{Z}Y,\alpha\times\beta)\to q_*(Y,\beta)$, we obtain the
natural map
\[q^*(X\times_ZY,\alpha\times\beta)\to (Y,\beta)\]
by adjunction. Let $(W,\delta)$ be the cokernel of this map. Since
$j^*$ is right exact and
\[j^*q^*(X\times_ZY,\alpha\times\beta)\cong
i^*p^*(X\times_ZY,\alpha\times\beta)\cong (Z,\gamma)\cong
j^*(Y,\beta),\] we have $j^*(W,\delta)=0$. That is,
$W=\delta(J)W$. But $J^2=0$, then $W=0$.

\medskip

(ii) follows from (i) immediately.

\medskip

Next we prove (iii). Let $(X,\gamma)\in\calC_{A'''}$. It is clear
that the kernel of the map $(X,\gamma)\to G(F(X,\gamma))$ is
$(\gamma(I)X\cap\gamma(J)X,\gamma)$. To prove the assertion, it is
enough to show that this map is surjective. Let $(Z,\delta)$ be the
cokernel so that we have the right exact sequence
\[(X,\gamma)\to G(F(X,\gamma))\to (Z,\delta)\to 0.\]
Since $F\circ G\cong\on{Id}$, by applying the right exact functor
$F$ to this sequence, we obtain that $F(Z,\delta)=0$. From the
definition of $F$, we have $Z=\delta(J)Z=\delta(I)Z$. But from
$J^2=0$, we obtain that $Z=0$.

\medskip

Finally, we prove (iv). First observe that as $R$-modules, $I$ and
$J$ are submodules of $A'''$, satisfying $I\cap J=0$. For any
$X\in\calR$, we have $f^*X=X\otimes_RA'''$, with the action of
$A'''$ denoted by $\alpha$. Then $\alpha(I)(X\otimes_RA''')$ (resp.
$\alpha(J)(X\otimes_RA''')$) is the image of the map $X\otimes_RI\to
X\otimes_RA'''$ ($X\otimes_RJ\to X\otimes_RA'''$). Therefore,
$\alpha(I)(X\otimes_RA''')\cap\alpha(J)(X\otimes_RA''')=0$. And (iv)
follows from (iii).
\end{proof}

\begin{lem}\label{res}Let $f:A\to B$ be a ring homomorphism, and $\calC$ an
$A$-linear abelian category. Then for any $X\in\calC_B$, the
adjunction map $f^*f_*X\to X$ is an epimorphism.
\end{lem}
\begin{proof} Let $Z$ be the cokernel of $f^*f_*X\to X$. Then
since $f_*$ is exact, $f_*Z$ is the cokernel of $f_*f^*f_*X\to
f_*X$. The natural adjunction $f_*X\to f_*f^*(f_*X)$ gives a
splitting of above morphism. Therefore, $f_*Z=0$, which implies
that $Z=0$.
\end{proof}

To complete the proof of the proposition, one only needs to apply
Lemma \ref{criterion} with $H=f^*f_*$.\end{proof}

Let $e:S\to\GL(\calC)$ be the unit map. We define
\[\gl(\calC):=T_e\GL(\calC).\]
\begin{cor}The stack $\gl(\calC)$ is a strictly
commutative $\calO_S$-linear Picard stack over $S$.
\end{cor}

\begin{ex}If $\calC=\calQ coh(X)$ for $f:X\to S$ separated and quasi-compact over $S$, according to Theorem \ref{Aut}, as $\calO_S$-linear Picard stacks,
\[\gl(\calC)\cong T_{[\calO_X]}\calP ic_X\oplus \calD er_S(\calO_X).\]
In addition, one can easily see that $(T_{[\calO_X]}\calP
ic_X)^\flat\cong \tau_{\leq 0}Rf_*\calO_X[1]$ under the dictionary
in \S \ref{dic1}.
\end{ex}

We would like give a concrete description of the $\calO_S$-linear
Picard stack structure of $\gl(\calC)$. For simplicity, let us
assume that $S=\spec R$ is affine.

Let $z:R\to D=R[\varepsilon]/\varepsilon^2$ be the structural map,
and $p:D\to R$ defined by $p(\varepsilon)=0$. Then
$p_*:\calC_R\to\calC_D$ realize $\calC_R$ as a full subcategory of
$\calC_D$. Let $\xi\in\gl(\calC)(S)$. So $\xi$ is a $D$-linear
auto-equivalence of $\calC_D$, such that for any $(X,d_X)\in\calC_D$
(where we use the notation as in Example \ref{dual numbers}), if we
write $\xi(X,d_X)=(\xi(X),\xi(d_X))$, then there is a canonical
isomorphism $\on{coker}\xi(d_X)\cong\on{coker}d_X$. It is easy to
see that then there is a canonical isomorphism
$\xi|_{p_*(\calC_R)}\cong\on{Id}$. Without loss of any generality,
we could assume that $\xi|_{p_*(\calC_R)}=\on{Id}$.

Let us consider $z^*X$ for $X\in\calC_R$. Recall by definition
$z^*X$ is the object $X\oplus X$ in $\calC_R$ with the action
$d_{z^*X}=\begin{pmatrix}0&1\\0&0\end{pmatrix}$. We thus have a
short exact sequence in $\calC_D$
\[
0\to p_*X\to z^*X\to p_*X\to 0.
\]
Observe that if one views this exact sequence in $\calC_R$ (i.e. one
applies $z_*$ to it), it is just
\[
0\to X\stackrel{(0,\on{id})}{\to} X\oplus X\to X\to 0,
\]
and therefore splits. Now if one applies $z_*\circ\xi$ to it, we
therefore obtain an exact sequence in $\calC_R$
\[0\to X\to z_*\xi z^*X\to X\to 0\]
Observe that in general, this short exact sequence does not split
even in $\calC_R$.

Let us define a strictly commutative Picard stack
$\underline{\Ext}(\on{Id}_\calC,\on{Id}_\calC)$ as follows. Its
objects in $\underline{\Ext}(\on{Id}_\calC,\on{Id}_\calC)(\spec
R')$ are self extensions of the identity functor of $\calC_{R'}$,
i.e. rules that functorially assign every $X\in\calC_{R'}$ a
self-extension. Its morphisms are the isomorphisms of these
self-extensions. It is a strictly commutative Picard stack under the Baer sums.
What we just discussed above shows that there is
a 1-morphism of stacks
\[\Pi:\gl(\calC)\to\underline{\Ext}(\on{Id}_\calC,\on{Id}_\calC)\]
\begin{prop}$\Pi$ is naturally a 1-homomorphisms
of the Picard stacks.
\end{prop}

\begin{proof}We show that if $\xi,\eta\in\gl(\calC)(R')$, then there are functorial isomorphisms of short exact sequences in $\calC_{R'}$
\[0\to X\to z_*(\xi+\eta)z^*X\to X\to 0,\]
and the Baer sum of
\[0\to X\to z_*\xi z^*X\to X\to 0, \quad\quad 0\to X\to z_*\eta z^*X\to X\to 0.\]
We leave it to readers to verify the compatibility of these isomorphisms.

We could assume that $R'=R$.  Recall the functors $\tilde{F}$ and $F,G$ defined in the course of proof of Proposition \ref{cond(E)},
\[\tilde{F}:\GL_{D\times_RD}(\calC_{D\times_RD})\to\GL_D(\calC_D)\times_{\GL_R(\calC_R)}\GL_D(\calC_D),\]
\[\calC_{D\times_R D}\substack{G\\ \leftrightharpoons\\F}\calC_D\times_{\calC_R}\calC_D.\]
Fix a quasi-inverse of $\tilde{F}$, denoted by $\tilde{G}$. Let
$g=\tilde{G}((\xi,\eta))$. Then by tracking of the proof of
Proposition \ref{cond(E)}, we find that in $\calC_{D\times_R D}$
there is a canonical isomorphism
\[g(i^*X)\cong G((\xi z^*X,\eta z^*X)),\]
where $i: R\to D\times_R D$ is the structural map. In addition, let
us recall that $i_*G(\xi z^*X,\eta z^*X)$ fits into the following
pullback diagram in $\calC_R$,
\[\begin{CD}
0@>>>X\oplus X@>>>i_*G(\xi z^*X,\eta z^*X)@>>>X@>>>0\\
@.@|@VVV@VV\Delta V@.\\
0@>>>X\oplus X@>>>z_*\xi z^*X\oplus z_*\eta z^*X@>>>X\oplus X@>>>0.
\end{CD}\]
where $\Delta$ is the diagonal map. Now let $(+):D\times_RD\to D$ be
the addition map as definition in (\ref{addition}). Then
$\xi+\eta=(+)^*(g)\in\GL_D(\calC_D)$. Observe that $(+)\circ i=z$.
Then by Lemma \ref{f^*F=(f^*F)f^*},
\[(\xi+\eta)(z^*X)=(+)^*(g)(z^*X)\cong (+)^*(g)((+)^*i^*X)\cong (+)^*(g(i^*X)).\]
Observe that $z_*(+)^*(g(i^*X))\cong z_*(+)^* G(\xi z^*X,\eta z^*X)$
fits into the following push-out diagram
\[\begin{CD}
0@>>>X\oplus X@>>>i_*G(\xi z^*X,\eta z^*X)@>>>X@>>>0\\
@.@VV+V@VVV@|@.\\
0@>>>X@>>>z_*(+)^* G(\xi z^*X,\eta z^*X)@>>>X@>>>0.
\end{CD}\]
The proposition follows.
\end{proof}

A direct consequence of this proposition (together with Proposition
\ref{same group structure}) is that for
$\xi,\eta\in\gl(\calC)\subset\GL_D(\calC_D)$, the short exact
sequence
\[0\to X\to z_*(\xi\circ \eta)z^*X\to X\to 0\]
is canonically isomorphic to the Baer sum of
\[0\to X\to z_*\xi z^*X\to X\to 0, \quad\quad 0\to X\to z_*\eta z^*X\to X\to 0.\]

\begin{rmk}\label{compare}
Let us have a closer look at the relation between $\gl(\calC)$ and
$\underline{\Ext}(\on{Id}_\calC,\on{Id}_\calC)$. According to Lemma
\ref{res}, an object $g\in\GL_D(\calC_D)$ is uniquely determined by
its restriction to the full subcategory of $\calC_D$ consisting of
objects of the form $z^*X, X\in\calC_R$. Then one would expect that
an object $E\in\underline{\Ext}(\on{Id}_\calC,\on{Id}_\calC)$ would
determine an object $g_E\in\GL_D(\calC_D)$ by sending $z^*X$ to
$E(X)$, where $E(X)$ fits into the short exact sequence $0\to X\to
E(X)\to X\to 0$ determined by $E$ (observe that $E(X)$ is naturally
an object in $\calC_D$). However, this is not correct. The point is
that if $u$ a morphism in $\calC_D$ between $z^*X$ and $z^*Y$. Then
\[z_*u:z_*z^*X\cong X\oplus X\to z_*z^*Y\cong Y\oplus Y\]
is generally of the form $\begin{pmatrix}u_1&0\\
u_2&u_1\end{pmatrix}$. On the other hand, the sought-after
auto-equivalence $g_E$ could only reasonably be defined on those
morphisms from $z^*X$ to $z^*Y$ of the forms
$u=\begin{pmatrix}u_1&0\\ 0&u_1\end{pmatrix}$. Namely, it sends such
$u$ to $E(u_1):E(X)\to E(Y)$.
\end{rmk}

\medskip

Now we turn to show that $\gl(\calC)$ indeed is a Lie 2-algebra. We should
prove that
\begin{prop}\label{goodness}The group $\GL(\calC)$ is good.
\end{prop}
\begin{proof}Let us first give a criterion of goodness for general 2-groups. Let $\calG$ be a
2-group over $S$. Observe that over the underline topological
space $\on{sp}(S)$ of $S$, we have the fiber product of sheaves of
$\calO_S$-algebras
\[\begin{CD}D_S(\calO_S))\cong\calO_S[\varepsilon]/\varepsilon^2@>q(\varepsilon)=\varepsilon_1\varepsilon_2>>D_S(\calO_S)\otimes_{\calO_S}D_S(\calO_S)\cong\calO_S[\varepsilon_1,\varepsilon_2]/(\varepsilon_1^2,\varepsilon_2^2)\\
@Vp(\varepsilon)=0VV@VVj(\varepsilon_1\varepsilon_2)=0V \\
\calO_S@>i>>D_S(\calO_S\oplus\calO_S)\cong\calO_S[\varepsilon_1,\varepsilon_2]/(\varepsilon_1^2,\varepsilon_2^2,\varepsilon_1\varepsilon_2),
\end{CD}\]
which gives rise to the following 1-commutative diagram
\[\begin{CD}
T\calG@>q>>\underline{\Hom}(I_S\times_S I_S,\calG)\\
@VVpV@VVjV\\
\calG@>i>>\underline{\Hom}(I_S(\calO_S\oplus\calO_S),\calG).
\end{CD}\]
\begin{lem}If $\calG$ satisfies the Condition (E), and the above 1-commutative diagram is Cartesian, then $\calG$ is
good.
\end{lem}
\begin{proof}Since $\calG$ satisfies the Condition (E), from the
above Cartesian diagram, one can easily obtain the following
Cartesian diagram for any free $\calO_S$-module $M$ of finite rank
\[\begin{CD}
\underline{\Hom}(I_S(M),\calG)@>>>\underline{\Hom}(I_S\times_S I_S(M),\calG)\\
@VVV@VVV\\
\calG@>>>\underline{\Hom}(I_S(\calO_S\oplus M),\calG).
\end{CD}\]
Indeed, assume that $M=\calO_S^r$. Then one readily checks that
the 1-morphism $\underline{\Hom}(I_S\times_S
I_S(M),\calG)\to\underline{\Hom}(I_S(\calO_S\oplus M),\calG)$ is
canonically isomorphic to the 1-morphism $(j\times\cdots\times j)$
from the $r$-folded product
\[\underline{\Hom}(I_S\times_SI_S,\calG)\times_{T\calG}\cdots\times_{T\calG}\underline{\Hom}(I_S\times_SI_S,\calG)\]
to the $r$-folded product
\[
\underline{\Hom}(I_S(\calO_S\oplus\calO_S),\calG)\times_{T\calG}\cdots\times_{T\calG}\underline{\Hom}(\calO_S\oplus\calO_S,\calG).
\]
The above Cartesian diagram follows easily.

Observe the kernel of the left column of the above commutative
diagram is $\frakg(M)$ and the kernel of the right column is
$T\frakg(M)$. Therefore, the canonical 1-homomorphism of
$\calO_S$-linear Picard stacks $\frakg(M)\to T\frakg(M)$ is a
1-isomorphism. Finally, observe that
$\frakg\otimes_{\calO_S}T\calO_S(M)\cong\frakg(M)$ since $\calG$
satisfies the Condition (E).
\end{proof}

By this lemma, the proof of this proposition will not be much
different from the proof of Proposition \ref{cond(E)}. Clearly, we
need only prove that if $\calC$ is an abelian category over
$S=\spec R$, then
\[
\GL_D(\calC_D)\stackrel{p^*\times q^*}{\to}\GL_{R}(\calC_{R})\times_{\GL_{D\times_RD}(\calC_{D\times_RD})}\GL_{D\otimes_R D}(\calC_{D\otimes_RD})
\]
is an isomorphism, where $D=R[\varepsilon]/(\varepsilon^2)$.

Recall that we have the pair of adjoint functors
\[
\calC_D\quad\substack{F\\ \rightleftharpoons\\G}\quad\calC_{R}\times_{\calC_{D\times_RD}}\calC_{D\otimes_RD,
}\]
where $F=p^*\times q^*$ and $G$ is the right adjoint of $F$. We
claim that the natural adjunctions $\on{Id}\to GF, FG\to\on{Id}$
are isomorphisms. This will imply the proposition

We first prove the isomorphism $\on{Id}\cong GF$. Indeed, this
follows easily from the fact that when $D\otimes_RD$ is regarded
as a $D$-module via the map $q$, we have $D\otimes_RD\cong D\oplus
R\oplus R$. Let $(X,d_X)$ be an object in $\calC_D$ (where we use
notation as in Example \ref{dual numbers}, so that $\varepsilon$
acts on $X$ via $d_X$), then the underlying objects of
$p^*(X,d_X), q^*(X,d_X)$ and $i^*p^*(X,d_X)\cong j^*q^*(X,d_X)$ in
$\calC_R$ are just $\on{coker}d_X, (X\oplus\on{coker}d_X\oplus
\on{coker}d_X)$ and $(\on{coker}d_X\oplus\on{coker}d_X\oplus
\on{coker}d_X)$ respectively. Therefore, the underlying object in
$\calC_R$ of $GF(X,d_X)$ fits into the following pullback diagram
\[\begin{CD}
GF(X,d_X)@>>>(X\oplus\on{coker}d_X\oplus \on{coker}d_X)\\
@VVV@VVV\\
\on{coker}d_X@>>>(\on{coker}d_X\oplus\on{coker}d_X\oplus
\on{coker}d_X).
\end{CD}\]
If one tracks carefully the morphisms, he will obtain that
$GF(X,d_X)\cong (X,d_X)$.

Next we prove that $FG\cong\on{Id}$. Let
$((X,\alpha),(Y,\beta),\varphi)\in\calC_{R}\times_{\calC_{D\times_RD}}\calC_{D\otimes_RD}$,
where notations are the same as in the proof of Theorem
\ref{cond(E)}. We have shown in Lemma \ref{criterion} (i) that
there is a natural isomorphism
\[p^*(X\times_ZY,\alpha\times\beta)\cong(X,\alpha),\]
and a natural surjective map
\[q^*(X\times_ZY,\alpha\times\beta)\twoheadrightarrow(Y,\beta),\]
since the map $j:D\otimes_RD\to D\times_RD$ is surjective with
square zero kernel. To finish the prove, we need to show that the
map $q^*(X\times_ZY,\alpha\times\beta)\twoheadrightarrow(Y,\beta)$
is an isomorphism. Let $(W,\delta)$ be its kernel. Then from the
properties of the pull-back, $(W,\delta\circ j)$ is the kernel of
the map $GFG((X,\alpha),(Y,\beta),\varphi)\to
G((X,\alpha),(Y,\beta),\varphi)$. But since $\on{Id}\cong GF$,
$W=0$.
\end{proof}

We thus obtain
\begin{thm}Let $\calC$ be a quasi-coherent sheaf of abelian categories over
$(\mathbf{Aff}/S)_{fppf}$. Then $\gl(\calC)$ is a pseudo Lie
2-algebra.
\end{thm}

\medskip

\noindent\bf Question: \rm give an explicit description of the Lie
bracket of $\gl(\calC)$.

\begin{rmk}In \cite{LV}, the Hochschild cohomology of an abelian category is introduced. They associated every $k$-linear abelian category $\calA$ a complex $C_{\on{ab}}(\calC)$. While the whole complex has a structure as a $B_\infty$-algebra, the truncation $\tau_{\leq 0}(C_{\on{ab}}(\calA)[1])$ is has a natural strucutre as a 2-term $L_\infty$-algebra, and therefore corresponds to a Lie 2-algebra. It seems that this Lie 2-algebra coincides with $\gl(\calC)$, at least in some cases. For example, if $\calC$ is the category of quasi-coherent sheaves on a smooth quasi-projective scheme $X$ over a field $k$, then the first two Hochschild cohomology of $\calC$ are
\[HH^0(\calC)\cong HH^0(X)\cong H^0(X,\calO_X)\]
\[HH^1(\calC)\cong HH^1(X)\cong H^1(X,\calO_X)\oplus H^0(X,\calT_X)\]
On the other hand, we know that in this case $H^0(\gl(\calC))=H^1(X,\calO_X)\oplus H^0(X,\calT_X)$ and $H^{-1}(\gl(\calC))=H^0(X,\calO_X)$.
\end{rmk}

\section{Appendix: The dictionary}\label{dic} We have the following
dictionaries.

\subsection{Strictly commutative $\calR$-linear
Picard stack v.s. 2-term complex of $\calR$-modules}\label{Pic}

Let $\calT$ be a topos, and $\calR$ be a ring in $\calT$. If
$\calP_1,\calP_2$ are two (strictly commutative) Picard stack over
$\calT$, we use $\underline{\Hom}_{\on{pic}}(\calP_1,\calP_2)$ to
denote the (strictly commutative) Picard stack of 1-homorphisms
from $\calP_1$ to $\calP_2$ over $\calT$ (cf. \cite{Del} \S
1.4.7). Let us first give the following definitions.

\begin{dfn}\label{R-linear Pic}A strictly commutative $\calR$-linear Picard stack
consists of:

\begin{enumerate}
\item[(a)]A strictly commutative Picard stack $\calP$, together
with an action of $\calR$ on $\calP$. In particular, we obtain a
1-morphism of stacks $\calR\to\underline{\Hom}(\calP,\calP)$. We
can harmless assume that $1\in\calR$ will give the identity
1-morphism of $\calP$. For any $r,r'\in\calR$ and $x\in\calP$, let
us denote $a$ to be the canonical isomorphism
$a_{r,r'}:r(r'(x))\cong (rr')(x)$, which satisfies the usual
compatibility conditions.

\item[(b)] A 1-morphism of stacks
$\calR\to\underline{\Hom}_{\on{pic}}(\calP,\calP)$, such that the
composition
$\calR\to\underline{\Hom}_{\on{pic}}(\calP,\calP)\to\underline{\Hom}(\calP,\calP)$
is identical to the 1-morphism in (a). For any
$r\in\calR,x,y\in\calP$, let us denote $b$ to be the canonical
isomorphism $b:r(x+y)\cong rx+ry$, (which respects the
associativity and commutativity constraints).

\item[(c)]The 1-morphism
$\calR\to\underline{\Hom}_{\on{pic}}(\calP,\calP)$ is a
1-homomorphism of Picard stacks, where the Picard groupoid
structure of $\calR$ comes from its addition. For any
$r,r'\in\calR, x\in\calP$, let us denote $c$ to be the canonical
isomorphism $c:(r+r')x\cong rx+r'x$, (which respects the
associativity and commutativity constraints).
\end{enumerate}

These data should satisfy the following compatibility conditions.

\begin{enumerate}
\item[(i)]The following diagrams commutes:
\[
\xymatrix{r(r'(x+y))\ar^{r(b)}[r]\ar^{a}[d]&r(r'x+r'y)\ar^{b}[r]&r(r'x)+r(r'y)\ar^{a}[d]\\
(rr')(x+y)\ar^{b}[rr]&& (rr')x+(rr')y.}
\]

\item[(ii)]The following diagram commutes:
\[
\xymatrix{
&(r+r')(x+y)\ar^{b}[dr]\ar_{c}[dl]&\\
r(x+y)+r'(x+y)\ar_{b}[d]&&(r+r')x+(r+r')y\ar^{c}[d]\\
(rx+ry)+(r'x+r'y)\ar^{\cong}[rr]&&(rx+r'x)+(ry+r'y). }
\]

\item[(iii)]The following diagram commutes:
\[
\xymatrix{
(r+r')(r''x)\ar^{c}[rr]\ar^{a}[d]&&r(r''x)+r'(r''x)\ar^{a}[d]\\
((r+r')r'')x\ar^{c}[rr]&&(rr'')x+(r'r'')x. }
\]

\item[(iv)]The following diagram commutes:
\[
\xymatrix{
r((r'+r'')x)\ar^{r(c)}[r]\ar^{a}[d]&r(r'x+r''x)\ar^{b}[r]&r(r'x)+r(r''x)\ar^{a}[d]\\
(r(r'+r''))x\ar^{c}[rr]&&(rr')x+(rr'')x. }
\]
\end{enumerate}
\end{dfn}

The definition is complicated. The point is it guarantees all the
other compatibility conditions in the spirit of Mac Lane's "coherence theorem".
The strictly commutative $\calR$-linear Picard stacks form a
2-category, whose morphisms we are going to define next.

\begin{dfn}\label{hom}
Let $\calP_1,\calP_2$ be two strictly commutative $\calR$-linear
Picard stacks over $\calT$. Then an $\calR$-linear 1-homomorphism
from $\calP_1$ to $\calP_2$ consists of a 1-homomorphism of Picard
stacks $F:\calP_1\to\calP_2$ and a 2-isomorphism $\tau$ of the
following 1-commutative diagram
\[
\xymatrix{\calR\times\calP\ar^{\on{act}}[rr]\ar_{\on{Id}\times F}[d]\ar^{\tau}@{==>}[drr]&&\calP\ar^{F}[d]\\
\calR\times\calP\ar^{\on{act}}[rr]&&\calP,}
\]
such that $\tau$ is compatible with those $a,b,c$'s. All the $\calR$-linear
1-homomorphisms from $\calP_1$ to $\calP_2$ form a groupoid,
denoted by $\Hom_{\calR-\on{pic}}(\calP_1,\calP_2)$, where the
2-morphisms between $(F,\tau),(F',\tau')$ are are natural
isomorphisms of $F$ and $F'$ compatible with $\tau,\tau'$. In
addition, one defines the stack
\[
\underline{\Hom}_{\calR-\on{pic}}(\calP_1,\calP_2),
\]
whose groupoid over $U\in\calT$ is
$\Hom_{\calR_U-\on{pic}}((\calP_1)_U,(\calP_2)_U)$.
\end{dfn}

It is clear that
$\underline{\Hom}_{\calR-\on{pic}}(\calP_1,\calP_2)$ indeed has a
natural structure as $\calR$-linear Picard stack over $\calT$.
Namely, if
$F_1,F_2\in\underline{\Hom}_{\calR-\on{pic}}(\calP_1,\calP_2)$,
then $(F_1+F_2)(x):=F_1(x)+F_2(x)$ and $(rF_1)(x):=r(F(x))$ for
any $x\in\calP_1$. All the constraints are clear.

\medskip

Let us state the following proposition, which is a generalization
of Corollary 1.4.17 of \cite{Del}. The proof is the same as the
original proof. Let
$\calK^\bullet:=\calK^{-1}\stackrel{d}{\to}\calK^0$ be a 2-term
complex of sheaves of $\calR$-modules. Then one can define a
strictly commutative $\calR$-linear Picard pre-stack
$\on{pch}(\calK^\bullet)$ as follows. For any $U\in\calT$, the
objects of $\on{pch}(\calK^\bullet)(U)$ sections $U\to\calK^0$,
and the morphisms from $x$ to $y$ for $x,y\in\calK^0(U)$ are the
sections $t\in\calK^{-1}(U)$ such that $d(t)=y-x$. The
$\calR$-linear structure is clear. Let $\on{ch}(\calK^\bullet)$
denote its stackification. Then $\on{ch}(\calK^\bullet)$ is a
strictly commutative $\calR$-linear Picard stack. Let
$\tilde{C}^{[-1,0]}(\calR\Mod)$ denote the 2-category whose
objects are two term complexes of $\calR$-modules
$\calK^{-1}\to\calK^0$ with $\calK^{-1}$ injective, 1-morphisms
are chain maps and 2-morphisms are homotopies between chain maps.
\begin{prop}\label{dic1}The assignment $\calK^\bullet\to\on{ch}(\calK^\bullet)$
can be upgraded two a functor from the 2-category
$\tilde{C}^{[-1,0]}(\calR\Mod)$ to the 2-category of strictly
commutative $\calR$-linear Picard stacks. Furthermore, the functor
gives rise to an equivalence of 2-categories.
\end{prop}

Let us fix an inverse functor $()^\flat$ of the above equivalence.
So for $\calP$ a strictly commutative $\calR$-Picard stack, we
have a 2-term complex of $\calR$-modules
$\calP^\flat:=\calK^{-1}\to\calK^0$.

As in \cite{Del} \S 1.4.18, and \S 1.4.19 one has
\begin{prop}$(\underline{\Hom}_{\calR-\on{pic}}(\calP_1,\calP_2))^\flat\cong\tau_{\leq
0}R\underline{\Hom}(\calP_1^\flat,\calP_2^\flat)$.
\end{prop}

\begin{prop} Let $f:(\calT_1,\calR_1)\to (\calT_2,\calR_2)$ be a
morphism of ringed topoi. Then for a strictly commutative
$\calR_1$-linear Picard stack $\calP$,
\[
(f_*\calP)^\flat\cong \tau_{\leq 0}Rf_*(\calP^\flat).
\]
\end{prop}

\begin{dfn}Let $\calP_1,\calP_2,\calP$ are three strictly commutative
$\calR$-linear Picard stack. An $\calR$-bilinear 1-homomorphism is
an $\calR$-linear 1-homomorphism
\[
F:\calP_1\to\underline{\Hom}_{\calR-\on{pic}}(\calP_2,\calP).
\]
$\calR$-bilinear 1-homomorphisms from a strictly commutative
$\calR$-linear Picard stacks, denoted by
$\underline{\Hom}_{\calR-\on{pic}}(\calP_1,\calP_2;\calP)$.

A $d$-uple $\calR$-linear 1-homomorphisms from
$\calP_1\times\cdots\times\calP_d$ to $\calP$ is defined by
induction as an $\calR$-linear 1-homomorphism
\[
F:\calP_1\to\underline{\Hom}_{\calR-\on{pic}}(\calP_2,\ldots,\calP_d;\calP).
\]
The $d$-uple $\calR$-linear 1-homomorphisms from
$\calP_1\times\cdots\times\calP_d$ to $\calP$ form a strictly
commutative $\calR$-linear Picard stack, denoted by
$\underline{\Hom}_{\calR-\on{pic}}(\calP_1,\calP_2,\ldots,\calP_d;\calP)$.
\end{dfn}

\begin{rmk}Another equivalent way to define the $\calR$-bilinear
1-homomorphism is: a 1-morphism
\[
F:\calP_1\times\calP_2\to\calP,
\]
and for any $x\in\calP_1, y\in\calP_2$, an $\calR$-linear
1-homomorphism structure on $F(x,-),F(-,y)$, such that for any
$x,y\in\calP_1,z,w\in\calP_2, r\in\calR$, the following diagrams
commute
\[
\xymatrix{
F(rx,z+w)\ar[r]\ar[d]&F(rx,z)+F(rx,w)\ar[d]& F(x+y,rz)\ar[r]\ar[d]&F(x,rz)+F(y,rw)\ar[d]\\
rF(x,z+w)\ar[r]&rF(x,z)+rF(x,w)& rF(x+y,z)\ar[r]& rF(x,z)+rF(y,z),}
\]
\[
\xymatrix{
F(x+y,z)+F(x+y,w)\ar[r]&(F(x,z)+F(y,z))+(F(x,w)+F(y,w))\ar^{\cong}[dd]\\
F(x+y,z+w)\ar[d]\ar[u]&\\
F(x,z+w)+F(y,z+w)\ar[r]& (F(x,z)+F(x,w))+(F(y,z)+F(y,w)).}
\]

Therefore, an $\calR$-bilinear 1-homomorphism will also give an
$\calR$-linear 1-homomorphism
\[
F':\calP_2\to\underline{\Hom}_{\calR-\on{pic}}(\calP_1,\calP)
\]
and vise versa.
\end{rmk}

As in \cite{Del} 1.4.20, we have

\begin{prop}Let $\calP_1,\calP_2$ be two strictly commutative
$\calR$-linear Picard stacks. Then there is a strictly commutative
$\calR$-linear Picard stack, denoted by
$\calP_1\otimes_\calR\calP_2$, together with a $\calR$-bilinear
1-homomorphism
\[
\otimes:\calP_1\times\calP_2\to\calP_1\otimes_\calR\calP_2,
\]
such that for any $\calR$-bilinear 1-homomorphism
$F:\calP_1\times\calP_2\to\calP$, there is a unique up to a unique
isomorphism pair $(F',\varepsilon)$, where $F'$ is an
$\calR$-linear 1-homomorphism
$F':\calP_1\otimes_\calR\calP_2\to\calP$ and $\varepsilon$ is a
2-isomorphism $F'(-\otimes-)\cong F$. Such pair
$(\calP_1\otimes_\calR\calP_2,\otimes)$ is unique up to a unique
1-isomorphism (up to isomorphisms). Indeed, one has
\[
(\calP_1\otimes_{\calR}\calP_2)^\flat\cong \tau_{\geq -1}(\calP_1^\flat\otimes^{\bbL}_{\calR}\calP_2^\flat).
\]
\end{prop}

\subsection{Lie 2-algebras v.s 2-term $L_\infty$-algebras}
\begin{dfn}\label{Lie 2-alg}A pseudo-Lie 2-algebra $\frakg$ over the ringed topos $(\calT,\calR)$ consists of:

(a) a strictly commutative $\calR$-linear Picard stack $\frakg$;

(b) an $\calR$-bilinear 1-homomorphism, called the Lie bracket
\[
[-,-]:\frakg\times\frakg\to\frakg.
\]

(c) A 2-isomorphism $s$ of the two $\calR$-bilinear
1-homomorphisms $\frakg\times\frakg\to\frakg$
\[
s_{x,y}:[x,y]\cong-[y,x] \quad\quad x,y\in\frakg.
\]

(d) A 2-isomorphism $j$ of the two $\calR$-trilinear
1-homomorphisms $\frakg\times\frakg\times\frakg\to\frakg$
\[j_{x,y,z}:[[x,y],z]+[[y,z],x]+[[z,x],y]\cong 0.\]
These data should satisfy the following compatibility conditions.

(i) The following diagram commutes
\[
\xymatrix{[x,y]\ar^{s_{x,y}}[r]\ar@{=}[d]&-[y,x]\ar^{-s_{y,x}}[d]\\
[x,y]&-(-[x,y])\ar^{a_{-1,-1}}[l], }\]
where $a_{r,r'}:r\circ r'\cong rr', r,r'\in\calR$ is as in Definition \ref{R-linear Pic}
(a).

(ii) The following diagram commutes
\[
\xymatrix{[[x,y],z]+[[y,z],x]+[[z,x],y]\ar^(.75){j_{x,y,z}}[rr]\ar^{\cong}[d]&& 0\ar@{=}[d]\\
[[y,z],x]+[[z,x],y]+[[x,y],z]\ar^(.75){j_{y,z,x}}[rr]&& 0. }
\]

(iii) The following diagram commutes
\[
\xymatrix{[[x,y],z]+[[y,z],x]+[[z,x],y]\ar^(.75){j_{x,y,z}}[rr]\ar_{[s_{x,y},z]+[s_{y,z,x]}+[s_{z,x},y]}[d]&&  0\ar@{=}[dd]\\
[-[y,x],z]+[-[z,y],x]+[-[x,z],y]\ar^\cong[d]&& \\
 -[[y,x],z]-[[x,z],y]-[[z,y],x]\ar^(.75){-j_{y,z,x}}[rr]&&0.}
\]

(iv) (\bf Jacobiator condition.\rm) To state the last compatibility condition, first observe that
$s,j$ together give a natural isomorphism
\[
j'_{x,y,z}:[[x,y],z]\to[x,[y,z]]-[y,[x,z]].
\]
Then one obtains a canonical isomorphism, called $f_{x,y,z,w}$
\[
\begin{array}{lcl}[[[x,y],z],w]&\stackrel{j'_{[x,y],z,w}}{\longrightarrow}&[[x,y],[z,w]]-[z,[[x,y],w]]\\
                            &\stackrel{j'_{x,y,[z,w]}-[z,j'_{x,y,w}]}{\longrightarrow}&([x,[y,[z,w]]]-[y,[x,[z,w]]])-([z,[x,[y,w]]]-[z,[y,[x,w]]]).
\end{array}
\]
Now we claim there are two canonical isomorphisms
\[
[[[x,y],z],w]+[[[y,z],x],w]+[[[z,x],y],w]\to 0.
\]
The first is
\[
[[[x,y],z],w]+[[[y,z],x],w]+[[[z,x],y],w]\longrightarrow[[[x,y],z]+[[y,z],x]+[[z,x],y],w]\stackrel{[j_{x,y,z},w]}{\longrightarrow}0.
\]
The second is
\[
\begin{array}{cl}&[[[x,y],z],w]+[[[y,z],x],w]+[[[z,x],y],w]\\
\stackrel{\sum f_{x,y,z,w}}{\longrightarrow}&\sum([x,[y,[z,w]]]-[y,[x,[z,w]]])-([z,[x,[y,w]]]-[z,[y,[x,w]]])\\
\longrightarrow &0,\end{array}
\]
where $\sum$ denotes the cyclic permutation of $x,y,z$. The fourth compatibility condition is that
these two isomorphisms coincide.

A pseudo Lie 2-algebra is called a Lie 2-algebra if there is a
functorial isomorphism
\[
s_x:[x,x]\cong 0.
\]
such that the following diagram commutes
\[
\xymatrix{[x,x]\ar^{s_{x,x}}[rr]\ar_{s_x}[dr]&&-[x,x]\ar^{-s_x}[dl]\\
&0& \ .}
\]
\end{dfn}

(Pseudo) Lie 2-algebras form a 2-category whose morphisms we are
going to define next.

\begin{dfn}\label{Lie hom}
Let $\frakg,\frakg'$ be two pseudo Lie 2-algebras. A
1-homomorphism $F:\frakg\to\frakg'$ consists of a 1-homomorphism
$F:\frakg\to\frakg'$ of $\calR$-linear Picard stacks and a
2-isomorphism of two $\calR$-bilinear 1-homomorphisms
\[
\theta: F([x,y])\cong [F(x),F(y)]
\]
such that the following two diagrams commute.
\[
\xymatrix{F([x,y])\ar^{\theta}[rr]\ar^{F(s)}[d]&&[F(x),F(y)]\ar^{s}[d]\\
F(-[y,x])\ar^{\tau}[r]&-F([y,x])\ar^{-\theta}[r]&-[F(y),F(x)], }
\]
\[
\xymatrix{F([[x,y],z]+[[y,z],x]+[[z,x],y])\ar^(0.75){F(j)}[r]\ar[d]&0\ar@{=}[d]\\
[[F(x),F(y)],F(z)]+[[F(y),F(z)],F(x)]+[[F(z),F(x)],F(y)]\ar^(0.9){j}[r]&0.}
\]
If both $\frakg$ and $\frakg'$ are Lie 2-algebras, $F$ is
called a Lie 2-algebra 1-homomorphism if in addition, the
following diagram commutes
\[
\xymatrix{F([x,x])\ar^{F(\theta)}[r]\ar^{F(s_x)}[d]&[F(x),F(x)]\ar^{s_{F(x)}}[d]\\
F(0)\ar^\cong[r]& 0.}
\]

A 2-morphism between two $(F,\theta), (F',\theta')$ are natural
isomorphisms between $F$ and $F'$ that are compatible with
$\theta$.
\end{dfn}

\begin{rmk}In \cite{BC}, Baez and Crans introduced another notion of Lie
2-algebras. They first introduced the notion of 2-vector spaces
over a field $k$, which from our point of view, are just those
strictly commutative $k$-linear Picard groupoids with all the
constraints (i.e. the associativity constraint, the commutativity
constraint, the unit constraint, and $a,b,c$ as in the Definition
\ref{R-linear Pic}) being identity maps. Then they defined the Lie
2-algebra as a 2-vector space, together with an asymmetric
bilinear morphism satisfying certain conditions. It is easy to see
that the Baez-Crans Lie 2-algebras, from our point of view, are
just those pseudo Lie 2-algebras such that the constraint $s$ is
the identity. Namely, the constraint $j$ in our Definition
\ref{Lie 2-alg} gives the jacobiator in the sense of Baez-Crans,
the compatibility condition (ii), (iii) in the Definition \ref{Lie
2-alg} amount to saying that $j$ is totally asymmetric, and the
jacobiator condition in the Definition \ref{Lie 2-alg} is
equivalent to the jacobiator identity given by Baez-Crans.
\end{rmk}

At this point, the definition of the 2-category of (pseudo) Lie
2-algebras is very complicated. However, according to Proposition
\ref{dic1}, a pseudo Lie 2-algebra should equivalent to a 2-term
complex of $\calR$-modules with additional structure. Not
surprisingly, this additional structure is just an
$L_\infty$-algebra structure on the complex. We should point out
Baez and Crans have already made this observation a precise
statement (cf. \cite{BC} Theorem 36) using their definition of Lie
2-algebras. We will improve their result in our more general
settings.

Let us recall the definition of $L_\infty$-algebras.

\begin{dfn}An $L_\infty$ algebra over $(\calT,\calR)$ is a graded
$\calR$-modules
\[\calL^\bullet=\oplus_k\calL^k\]
together with a system of $\calR$-module multilinear homomorphisms
\[l_n: (\calL^\bullet)^{\otimes n}\to\calL^\bullet, \quad 1\leq n<\infty\]
such that: (i) for any $n$, $\deg(l_n)=2-n$, and $l_n$ is totally
asymmetric, i.e. for homogeneous elements
$x_1,\ldots,x_n\in\calL^\bullet$,
\[l_n(x_{\sigma(1)},\ldots,x_{\sigma(n)})=\chi(\sigma)l_n(x_1,\ldots,x_n)\]
where $\sigma\in S_n$ is an element in the symmetric group and
$\chi(\sigma):=\chi(\sigma;x_1,\ldots,x_n)$ is defined via
\[x_{\sigma(1)}\wedge\cdots\wedge x_{\sigma(n)}=\chi(\sigma;x_1,\ldots,x_n)x_1\wedge\cdots\wedge x_n\]

(ii) for any $n$ and homogeneous elements
$x_1,\ldots,x_n\in\calL^\bullet$,
\[\sum_{i+j=n+1}\sum_{\substack{\sigma(1)<\cdots<\sigma(j)\\
\sigma(j+1)<\cdots<\sigma(n)}}(-1)^{ij}\chi(\sigma)l_i(l_j(x_{\sigma(1)},\ldots,x_{\sigma(j)}),x_{\sigma(j+1)},\ldots,x_{\sigma(n)})=0\]
\end{dfn}

We will be interested in the case where $\calL^\bullet$ is an
$L_\infty$-algebra with $\calL^i=0, i\neq -1,0$. It is explained
in \cite{BC} that 2-term $L_\infty$-algebras form a 2-category. We
will denote the subcategory consisting of those $\calL^\bullet$
such that $\calL^{-1}$ is injective.

Let $\calL^\bullet$ be a 2-term $L_\infty$-algebra. Then $l_1$
makes $\calL^\bullet$ a 2-term complex of $\calR$-modules. Then
$\on{ch}(\calL^\bullet)$ is a strictly commutative $\calR$-linear
Picard stack. $\on{ch}(l_2)$ gives $\on{ch}(\calL^\bullet)$ a
$\calR$-bilinear 1-homomorphism satisfying satisfying
$\on{ch}(l_2)(x,y)=-\on{ch}(y,x)$, and $\on{ch}(l_3)$ gives a
canonical isomorphism
\[\on{ch}(l_3):\on{ch}(\on{ch}(l_2)(x,y),z)+\on{ch}(\on{ch}(l_2)(y,z),x)+\on{ch}(\on{ch}(l_2)(z,x),y)\cong 0\]
It is easy to check that the strictly commutative $\calR$-linear Picard stack $\on{ch}(\calL^\bullet,l_1)$ together with the Lie bracket $\on{ch}(l_2)$ and the natural isomorphisms $\on{ch}(l_3)$ is indeed a
pseudo Lie 2-algebra in the sense of Definition \ref{Lie 2-alg}.

\begin{prop}\label{dic2}Assume that $\frac{1}{2}\in\calR$. Then the natural functor
$\calL^\bullet\mapsto\on{ch}(\calL^\bullet)$ is an equivalence
from the 2-category of 2-term $L_\infty$-algebras with $\calL^{-1}$ injective to the 2-category of Lie 2-algebras.
\end{prop}
\begin{proof}First observe that if $2$ is invertible in $\calR$, then every pseudo Lie 2-algebra has a natural structure as a Lie 2-algebra.

We only show that for any Lie 2-algebra $\frakg$,
there is an $L_\infty$-algebra $\calL^\bullet$ such that
$\on{ch}(\calL^\bullet)\cong\frakg$ as pseudo-Lie 2-algebras. Then
the proposition follows from \cite{BC}.

By applying the quasi-inverse $()^\flat$ of $\on{ch}$, we obtain a
2-term complex of $\calR$-modules $\frakg^\flat$ with an
$\calR$-bilinear morphism
\[\tilde{l}_2:=[-,-]^\flat:\frakg^\flat\times\frakg^\flat\to\frakg^\flat\]
\[s=s^\flat:\frakg^\flat\times\frakg^\flat\to\frakg^\flat[-1]\]
and $\calR$-trilinear morphism
\[\tilde{l}_3:=j^\flat:\frakg^\flat\times\frakg^\flat\times\frakg^\flat\to\frakg^\flat[-1]\]
satisfying:
\begin{enumerate}
\item $\tilde{l}_2(x,y)+\tilde{l}_2(y,x)=ds^\flat(x,y)$ for any
$x,y\in\frakg^\flat$, and $s^\flat(x,y)=s^\flat(y,x)$;

\item
$\tilde{l}_2(\tilde{l}_2(x,y),z)+\tilde{l}_2(\tilde{l}_2(y,z),x)+\tilde{l}_2(\tilde{l}_2(z,x),y)=d\tilde{l}_3(x,y,z)$;

\item$\tilde{l}_3(x,y,z)=\tilde{l}_3(y,z,x)$;

\item$\tilde{l}_3(x,y,z)+\tilde{l}_3(y,x,z)=\tilde{l}_2(s^\flat(x,y),z)+\tilde{l}_2(s^\flat(y,z),x)+\tilde{l}_2(s^\flat(z,x),y)$;

\item a very complicated formula that corresponds to the jacobiator condition in the definition of pseudo Lie 2-algebras.
\end{enumerate}
Let us define $l_1=d,
l_2(x,y)=\tilde{l}_2(x,y)-\frac{1}{2}ds^\flat(x,y)$ and
\[\begin{split}l_3(x,y,z)&=\tilde{l}_3(x,y,z)-\frac{1}{2}(\tilde{l}_2(s^\flat(x,y),z)+\tilde{l}_2(s^\flat(y,z),x)+\tilde{l}_2(s^\flat(z,x),y))\\
                         &-\frac{1}{2}(s^\flat(\tilde{l}_2(x,y),z)+s^\flat(\tilde{l}_2(y,z),x)+s^\flat(\tilde{l}_2(z,x),y))\\
                         &+\frac{1}{4}(s^\flat(ds^\flat(x,y),z)+s^\flat(ds^\flat(y,z),x)+s^\flat(ds^\flat(z,x),y))\end{split}\]

One can easily check that:
\begin{enumerate}
\item$l_2(x,y)+l_2(y,x)=0$;

\item$l_2(l_2(x,y),z)+l_2(l_2(y,z),x)+l_2(l_2(z,x),y)=dl_3(x,y,z)$;

\item$l_3(x,y,z)=l_3(y,z,x)$;

\item$l_3(x,y,z)+l_3(y,x,z)=0$.
\end{enumerate}
We can use $l_2$ to define a new Lie bracket on $\on{ch}(\frakg^\flat,d)$
\[[x,y]'=\on{ch}(l_2)(x,y).\]
Then $l_3$ gives a canonical 2-isomorphism of two $\calR$-trilinear 1-homomorphism
\[\on{ch}(l_3): [[x,y]',z]'+[[y,z]',x]'+[[z,x]',y]'\cong 0.\]
It is readily to check that $\on{ch}(\frakg^\flat,d)$ together with the Lie bracket $[-,-]'$ and the natural isomorphisms $\on{ch}(l_3)$ satisfies all the conditions in Definition \ref{Lie 2-alg} except the jacobiator condition. We claim that it also satisfies the jacobiator condition so that it is a Lie 2-algebra. To prove this, we need the following general lemma, whose proof is left to readers.

\begin{lem}Let $\frakg$ be a pseudo Lie 2-algebra, and $\frakg'$ be a strictly commutative $\calR$-linear Picard stack satisfying all the structures as a pseudo Lie 2-algebra except the jacobiator condition. If there is a surjective $\calR$-linear 1-homomorphism $F:\frakg\to\frakg'$ of Picard stacks together with a 2-isomorphism of two $\calR$-bilinear 1-homomorphisms
\[\theta: F([x,y])\cong[F(x),F(y)],\]
such that all the compatibility conditions in Definition \ref{Lie hom} hold, then $\frakg'$ satisfies the jacobiator condition and is therefore a pseudo Lie 2-algebra.
\end{lem}

Now, we apply the above lemma to the identity 1-homomorphism $\on{Id}:\on{ch}(\frakg^\flat)\to\on{ch}(\frakg^\flat)$. The first $\on{ch}(\frakg^\flat)$ is equipped with the Lie bracket given by $\on{ch}(\tilde{l}_2)$ and the natural isomorphisms $\on{ch}(\tilde{l}_3)$, and therefore is a pseudo Lie 2-algebra isomorphic to $\frakg$. The second $\on{ch}(\frakg^\flat)$ is equipped with the Lie bracket given by $\on{ch}(l_2)$ and the natural isomorphisms $\on{ch}(l_3)$, which satisfies all the structures as a pseudo Lie 2-algebra except the jacobiator condition. Now there is a canonical 2-isomorphism
\[\on{ch}(-\frac{1}{2}s^\flat):[x,y]'\to[x,y].\]
It is easy to show that it satisfies all the compatibility conditions in Definition \ref{Lie hom}. Therefore, we conclude by the above lemma that $\on{ch}(\frakg^\flat,d)$ together with  the Lie bracket given by $\on{ch}(l_2)$ and the natural isomorphisms $\on{ch}(l_3)$ is a pseudo Lie 2-algebra, which is isomorphic to $\frakg$ we begin with.

Finally, we can prove that $\frakg^\flat,l_1=d,l_2,l_3$ is a 2-term $L_\infty$-algebra. Namely, since $\on{ch}(\frakg^\flat,l_1=d,l_2,l_3)$ is a pseudo Lie 2-algebra, its jacobiator condition gives us the following identity
\[l_2(l_3(x,y,z),w)=\sum l_3(l_2(x,y),z,w)+l_3(x,y,l_2(z,w))-l_2(z,l_3(x,y,w)),\]
where the sum is taken over the cyclic permutations of $x,y,z$. This identity, together with the trivially satisfied identity
\[dl_3(x,y,z)=l_2(l_2(x,y),z)+l_2(l_2(y,z),x)+l_2(l_2(z,x),y)\]
implies that $\frakg^\flat,l_1=d,l_2,l_3$ is a 2-term $L_\infty$-algebra.
\end{proof}

\end{document}